\documentclass{article}
\usepackage{amsmath}
\usepackage{amsthm}
\usepackage{hyperref}
\usepackage[capitalise]{cleveref}
\usepackage{float}
\usepackage[utf8]{inputenc}
\usepackage[T1]{fontenc}
\usepackage{subcaption}
\usepackage{caption}
\usepackage{enumitem}
\usepackage{algorithm}
\usepackage{algpseudocode}
\usepackage{comment}
\usepackage{lipsum}
\usepackage{amsfonts}
\usepackage{graphicx}
\usepackage{epstopdf}
\usepackage{csvsimple}
\usepackage{tikz}
\usepackage{pgfplots}
\usepackage{pgfplotstable}
\pgfplotsset{compat=newest}
\usepackage{xcolor}
\usepackage{amsopn}
\usepackage{microtype}

\usepackage[giveninits=true,style = numeric,backend=biber,sorting=nyvt]{biblatex}
\DeclareNameAlias{default}{family-given}
\newcommand{\Krasnoselsky}{Krasnosel'ski\u{\i}}
\newcommand{\PrimS}{\ensuremath{\mathcal{H}}}
\newcommand{\DualS}{\ensuremath{\mathcal{K}}}
\newcommand{\nat}{\ensuremath{\mathbb{N}}}

\newcommand{\weakto}{\ensuremath{\rightharpoonup}}
\newcommand{\Id}{\ensuremath{\mathop{\mathrm{Id}}}}
\newcommand{\p}[1]{\ensuremath{\mathord{\left( #1 \right)}}}
\newcommand{\norm}[1]{\ensuremath{\mathord{\left\Vert #1 \right\Vert}}}
\newcommand{\set}[1]{\ensuremath{\mathord{\left\lbrace #1 \right\rbrace}}}
\newcommand{\setcond}[2]{\ensuremath{\mathord{\left\lbrace #1 : #2 \right\rbrace}}}
\newcommand{\inpr}[2]{\ensuremath{\mathord{\left\langle #1, #2 \right\rangle}}}

\newcommand{\normsq}[1]{\left\Vert #1 \right\Vert^2}

\newcommand{\zer}[1]{\operatorname{\mathrm{zer}}(#1)}
\newcommand{\reals}{\mathbb{R}}

\newcommand{\posop}[1]{\mathcal{M}\p{\PrimS}}
\newcommand{\lp}{\left(}
\newcommand{\rp}{\right)}
\newcommand{\xs}{x^\star}
\newcommand{\xwc}{x_{\mathrm{wc}}^\star}
\newcommand{\dist}[3]{\hspace{0.5mm}\mathrm{dist}_#1(#2,#3)}
\newcommand{\distsq}[3]{\hspace{0.5mm}\mathrm{dist}^2_#1(#2,#3)}
\newcommand{\gra}[1]{\operatorname{gra}(#1)}
\newcommand{\seq}[2]{$(#1_#2)_{#2\in\mathbb{N}}$}
\newcommand{\seqsc}[2]{$\p{#1}_{#2\in\mathbb{N}}$}

\newcommand{\defeq}{\ensuremath{\mathrel{\mathop:}=}}
\DeclareMathOperator{\argmin}{arg\,min}

\ifpdf
  \DeclareGraphicsExtensions{.eps,.pdf,.png,.jpg}
\else
  \DeclareGraphicsExtensions{.eps}
\fi

\theoremstyle{remark}
\newtheorem{remark}{Remark}
\newtheorem{example}{Example}

\crefname{hypothesis}{Hypothesis}{Hypotheses}

\theoremstyle{plain}
\newtheorem{assumption}{Assumption}
\newtheorem{lemma}{Lemma}
\newtheorem{theorem}{Theorem}
\newtheorem{corollary}{Corollary}
\crefname{assumption}{Assumption}{Assumptions}

\theoremstyle{definition}
\newtheorem{definition}{Definition}
\newtheorem*{keywords}{Key words}
\newtheorem*{AMS}{AMS Subject Classification (2020)}

\title{Forward--Backward Splitting with Deviations\\for Monotone Inclusions}

\author{Hamed Sadeghi\thanks{Email: \{\href{mailto:hamed.sadeghi@control.lth.se}{hamed.sadeghi}, \href{mailto:sebastian.banert@control.lth.se}{sebastian.banert}, \href{mailto:pontus.giselsson@control.lth.se}{pontus.giselsson}\}@control.lth.se. Affiliation: Lund University, Lund, Sweden.}
   \and Sebastian Banert\footnotemark[1]
   \and Pontus Giselsson\footnotemark[1]}

\ifpdf
\hypersetup{
  pdftitle={},
  pdfauthor={H. Sadeghi, S. Banert, and P. Giselsson}
}
\fi

\begin{document}
\maketitle

\begin{abstract}
We propose and study a weakly convergent variant of the forward--backward algorithm for solving structured monotone inclusion problems.
Our algorithm features a per-iteration deviation vector which provides additional degrees of freedom.
The only requirement on the deviation vector to guarantee convergence is that its norm is bounded by a quantity that can be computed online.
This approach provides great flexibility and opens up for the design of new and improved forward--backward-based algorithms, while retaining global convergence guarantees.
These guarantees include linear convergence of our method under a metric subregularity assumption without the need to adapt the algorithm parameters.

Choosing suitable monotone operators allows for incorporating deviations into other algorithms, such as Chambolle--Pock and \Krasnoselsky--Mann iterations.
We propose a novel inertial primal--dual algorithm by selecting the deviations along a momentum direction and deciding their size using the norm condition.
Numerical experiments demonstrate our convergence claims and show that even this simple choice of deviation vector can improve the performance, compared, e.g., to the standard Chambolle--Pock algorithm.
\end{abstract}

\begin{keywords}
  forward--backward splitting, monotone inclusions, global convergence, linear convergence rate, inertial forward--backward method, inertial primal--dual algorithm.
\end{keywords}

\begin{AMS}
47H05, 47N10, 49M29, 65K05, 90C25.
\end{AMS}

\section{Introduction}
\label{sec:intro}

\emph{Forward--backward} (FB) splitting \cite{Bruck1975AnIS,Lions_1979,Passty_1979} has been extensively used to solve structured monotone inclusion problems of finding $x \in \PrimS$ such that 
\begin{equation}\label{eq:main_inclusion}
    0\in Ax+Cx,
\end{equation}
where $A\colon\PrimS\to 2^\PrimS$ is a maximally monotone operator, $C\colon\PrimS\to\PrimS$ is a cocoercive operator, and $\PrimS$ is a real Hilbert space. The algorithm sequentially performs a forward step with the operator $C$ followed by a backward step with $A$ to arrive at the iteration
\begin{align}\label{eq:FB-standard}
    x_{n+1}=(\Id + \gamma_nA)^{-1}\circ(\Id-\gamma_nC)x_n,
\end{align}
where $\gamma_n > 0$ is a step-size parameter.

One of the most important special cases of this setting is first-order algorithms for convex optimization: let $f\colon \PrimS \to \reals$ be a convex, differentiable function whose gradient is Lipschitz continuous and $g\colon \PrimS \to \reals \cup \set{\pm \infty}$ be a proper, convex and lower semicontinuous function, and let $A = \partial g$ (the subdifferential of $g$) and $C = \nabla f$. Then, \eqref{eq:main_inclusion} is the problem of finding a minimizer of $f + g$, and \eqref{eq:FB-standard} describes the \emph{proximal gradient method} \cite{Combettes2011ProximalSM}.

In this paper, we present a weakly convergent extension to the standard FB splitting method to solve the monotone inclusion \eqref{eq:main_inclusion}. A simplified instance of our algorithm is given by
\begin{equation}\label{eq:simplified_instance}
    \begin{aligned}
    p_n&=(\Id+\tfrac{1}{\beta}A)^{-1}\circ(\Id - \tfrac{1}{\beta}C)(x_n+u_n)\\
    x_{n+1} &= p_n - u_n
\end{aligned}
\end{equation}
where $u_n$ is a \emph{deviation (vector)} and $\frac{1}{\beta}>0$ is a cocoercivity constant of $C$. By letting $u_n=0$, a step of \eqref{eq:simplified_instance} reduces to the standard FB step in \eqref{eq:FB-standard}. The addition of $u_n$ therefore gives added flexibility that can be utilized to improve performance. In order to ensure convergence of this algorithm, $u_n$ has to satisfy the \emph{norm condition}
\begin{align}\label{eq:simplified-norm-condition}
    \norm{u_{n}}^2 &\leq \tfrac{1-\epsilon}{4}\norm{p_{n-1}-x_{n-1}+u_{n-1}}^2,
\end{align}
where $\epsilon\in[0,1)$ is arbitrary and the quantity to the right-hand side of the inequality is computable online since the variables are known from previous iterations.

Safeguarding is a common technique to ensure global convergence in optimization algorithms, for instance the Wolfe conditions in line-search \cite[Chapter 3]{NocedalWright:2006} ensure a sufficient decrease in the objective function value, and trust-region methods \cite[Chapter 4]{NocedalWright:2006} are based on a quadratic model having sufficient accuracy within a given radius. Recently, a norm condition similar to \eqref{eq:simplified-norm-condition} has been combined with a deep-learning approach to speed up the convergence \cite{Banert2021AcceleratedFO}. Even for monotone operators, line-search strategies with safeguarding have been developed, see \cite[Eq. (2.4)]{Tseng:2000} for an example. In contrast to line search, \eqref{eq:simplified-norm-condition} does not require to compute (and possibly reject) several steps per iteration. For other examples of safeguarding, see \cite{giselsson2016line,sadeghi2021hybrid,themelis2019supermann,zhang2020globally}.

Our main algorithm (\cref{alg:main:deviations}) is more general than \eqref{eq:simplified_instance}. It uses two deviation vectors and a slightly more involved safeguard condition.
A similar algorithm with deviation vectors has been proposed in \cite{Banert2021AcceleratedFO} to extend the proximal gradient method for convex minimization.
The fact that we consider the more general monotone inclusion setting, allows us to apply our results, e.g., to the Chambolle--Pock~\cite{chambolle2011first} and Condat--V{\~u} \cite{condat2013primal,vu2013splitting} methods that both are preconditioned FB methods \cite{he2012convergence}. To facilitate the derivation of these special cases, we derive our algorithm with explicit preconditioning, such as in \cite{Chouzenoux_2013,Combettes_2012,Giselsson2019NonlinearFS-journal,Giselsson_2015,Giselsson2014DiagonalSI,Giselsson2014PreconditioningIF,pock2011diagonal,Raguet_2015}.

Our algorithm is also related to inexact FB methods, which are studied in the framework of monotone inclusions \cite{Raguet_2013,Solodov_2000,Solodov_2001,vu2013splitting} and in a convex optimization setting \cite{condat2013primal,Schmidt_2011_nips,Villa_2013}. By including error terms in the FB splitting algorithms, these works allow for inaccuracies in the forward and backward step evaluations. The convergence of the algorithm is usually based on a summability assumption on the error sequences and would therefore allow arbitrarily large errors as long as they only happen for a finite number of iterations. The idea behind our method is in stark contrast to these methods, as our method is designed for actively choosing the deviations with the aim to improve performance.

We instantiate our general scheme in three special settings; the standard FB setting, the primal-dual setting of Condat--V{\~u}, and the \Krasnoselsky--Mann setting. We also propose a further specialization of the primal-dual setting of Chambolle--Pock in which we select the deviations in a heavy-ball type \cite{polyak1964some} momentum direction (see \cite{sadeghi2022dwifob} for another novel usage of the deviations in a primal--dual setting). The resulting algorithm bears similarities with the inertial FB methods \cite{Alvarez_2000,Alvarez_2001,attouch2019convergence,Cholamjiak_2018,lorenz2015inertial} when applied in a primal-dual setting. Numerical experiments show improved performance of our method compared to Chambolle--Pock and a primal--dual version of Lorenz--Pock \cite{lorenz2015inertial}.

\paragraph{\textbf{Contributions.}} The most notable differences of this work to existing literature can be summarized as follows:
\begin{itemize}
    \item[$-$] Compared to the standard FB, we extend the degrees of freedom by allowing the input argument to the FB operator to deviated from a pre-specified point.
    
    \item[$-$] Unlike various known examples of momentum methods, the increase is not achieved with a fixed number of parameters, but the design parameter has the dimension of the underlying problem.

    \item[$-$] In contrast to inexact FB algorithms \cite{condat2013primal,Raguet_2013,vu2013splitting,Villa_2013}, the bound on the deviations is a scalar condition with known quantities in each step instead of a summability condition that has limited meaning for a finite number of steps.

    \item[$-$] In contrast to the deviation-based FB method for convex optimization in \cite{Banert2021AcceleratedFO}, our work considers more general monotone inclusion problems. Hence, we immediately obtain the algorithms of Chambolle--Pock~\cite{chambolle2011first} and \Krasnoselsky--Mann with deviations as special cases. Moreover, our convergence result is slightly stronger than \cite[Theorem 3.2]{Banert2021AcceleratedFO}. To the best of our knowledge, neither algorithm is a special case of the other.
    
    \item[$-$] In addition to showing weak convergence of our algorithm, we show that under a metric subregularity assumption the algorithm converges strongly to a point in the solution set of the problem with a linear rate of convergence.
    
    \item[$-$] As an example for the expressiveness of the deviation-based approach, we introduce a novel inertial primal--dual algorithm by selecting the deviations along a momentum direction---in the sense of Polyak \cite{polyak1964some}---and deciding their size using the norm condition. 
\end{itemize}

\paragraph{\textbf{Outline of the paper.}} The organization of the paper is as follows. In~\cref{sec:prelim}, we provide notations and  some definitions. In~\cref{sec:alg}, the proposed algorithm is introduced. In \cref{sec:conv}, we prove weak convergence of the method and linear and strong convergence under a metric subregularity assumption. In~\cref{sec:special_cases}, some special cases of the proposed algorithm are presented and \cref{sec:novel_inertial_pd_alg} further specializes one of these to arrive at a novel inertial primal--dual algorithm. We conclude the paper by presenting the numerical results in~\cref{sec:numerical_experiments}.

\section{Preliminaries}
\label{sec:prelim}
Throughout the paper, the set of real numbers is denoted by $\reals$; $\PrimS$ and $\DualS$ denote real Hilbert spaces that are equipped with inner products and induced norms, which are denoted by $\inpr{\cdot}{\cdot}$ and $\norm{\cdot}=\sqrt{\inpr{\cdot}{\cdot}}$, respectively. A bounded, self-adjoint operator $M\colon \PrimS\to \PrimS$ is said to be \emph{strongly positive} if there exists some $c > 0$ such that $\inpr{x}{Mx}\geq c\normsq{x}$ for all $x\in\PrimS$. We use the notation $\posop{\PrimS}$ to denote the set of linear, self-adjoint, strongly positive operators on $\PrimS$. For $M\in\posop{\PrimS}$ and for all $x,y\in\PrimS$, the $M$-induced inner product and norm are denoted by $\inpr{x}{y}_M=\inpr{x}{My}$ and $\norm{x}_M = \sqrt{\inpr{x}{Mx}}$, respectively.

Young's inequality
\[
  \inpr{x}{y} \leq \frac{\omega}{2} \norm{x}^2_M + \frac{1}{2 \omega} \norm{y}^2_{M^{-1}}
\]
holds for all $x, y \in \PrimS$, $\omega > 0$, and $M \in \posop{\PrimS}$. Hence, with the same variables,
\[
  \norm{x + y}_M^2 = \norm{x}_M^2 + \norm{y}_M^2 + 2 \inpr{x}{M y} \leq \p{1 + \omega} \norm{x}_M^2 + \frac{1 + \omega}{\omega} \norm{y}_M^2.
\]

Let $M\in\posop{\PrimS}$, $x\in\PrimS$, and $S\subset\PrimS$ be a nonempty closed convex set. The $M$-induced projection of $x$ onto the set $S$ is defined as $\Pi_{S}^Mx = \argmin_{y\in S}\norm{x-y}_M$, and the $M$-induced distance from $x$ to $S$ is defined by $\dist{M}{x}{S} = \norm{x-\Pi_{S}^Mx}_M$.

The notation $2^\PrimS$ denotes the \emph{power set} of $\PrimS$. A map $A\colon\PrimS\to 2^{\PrimS}$ is characterized by its graph $\gra{A} = \setcond{\p{x, u}\in\PrimS\times\PrimS}{u\in Ax}$. An operator $A\colon \PrimS\to 2^{\PrimS}$ is \emph{monotone} if $\inpr{u-v}{x-y}\geq0$ for all $\p{x, u}, \p{y, v}\in\gra{A}$.
A monotone operator $A\colon \PrimS \to 2^{\PrimS}$ is \emph{maximally monotone} if there exists no monotone operator $B\colon\PrimS\to 2^\PrimS$ such that $\gra{B}$ properly contains $\gra{A}$.  

Let $M\in \posop{\PrimS}$. An operator $T\colon\PrimS\to\PrimS$ is said to be 
\renewcommand{\labelenumi}{{(\roman{enumi})}}
\begin{enumerate}
    \item \emph{$L$-Lipschitz continuous} ($L \geq 0$) w.r.t.\ $\norm{\cdot}_M$ if
        \begin{equation*}
            \norm{Tx-Ty}_{M^{-1}}\leq L\norm{x-y}_M \qquad \text{for all } x,y\in\PrimS;
        \end{equation*}
    \item \emph{$\frac{1}{\beta}$-cocoercive} ($\beta > 0$) w.r.t. $\norm{\cdot}_M$ if
        \begin{equation*}
            \inpr{Tx -Ty}{x-y}\geq\frac{1}{\beta}\norm{Tx-Ty}_{M^{-1}}^2\qquad \text{for all } x,y\in\PrimS;
        \end{equation*}
    \item \emph{nonexpansive} if it is $1$-Lipschitz continuous w.r.t. $\norm{\cdot}$;
    \item \emph{firmly nonexpansive} if 
        \begin{equation*}
            \norm{Tx-Ty}^2 + \norm{\p{\Id-T}x-\p{\Id-T}y}^2\leq\norm{x-y}^2  \qquad \text{for all } x,y\in\PrimS.
        \end{equation*}
\end{enumerate}
By the Cauchy--Schwarz inequality, a $\frac{1}{\beta}$-cocoercive operator is $\beta$-Lipschitz continuous. The \emph{resolvent} of a maximally monotone operator $A\colon \PrimS \to 2^{\PrimS}$ is denoted by $J_{\gamma A} \colon \PrimS \to \PrimS$ and defined as $J_{\gamma A} \defeq (\Id+\gamma A)^{-1}$. $J_{\gamma A}$ has full domain, is firmly nonexpansive \cite[Corollary 23.8]{bauschke2017convex}, and is single-valued.

\section{Forward--backward splitting with deviations}
\label{sec:alg}

We consider structured monotone inclusion problems of the form
\begin{align} \label{eq:monotone_inclusion}
    0 \in Ax + Cx,
\end{align}
that satisfy the following assumptions.

\begin{assumption}
\label{assum:monotone_inclusion}
Assume that $\beta > 0$,
\renewcommand{\labelenumi}{\emph{(\roman{enumi})}}
\begin{enumerate}
    \item $A\colon \PrimS \to 2^{\PrimS}$ is maximally monotone.
    \item $C\colon \PrimS \to \PrimS$ is $\frac{1}{\beta}$-cocoercive with respect to $\norm{\cdot}_M$ with $M\in\posop{\PrimS}$.
    \item The solution set $\zer{A+C} \defeq \setcond{x\in\PrimS}{0\in Ax+Cx}$ is nonempty.
\end{enumerate}
\end{assumption}
Observe that, as a cocoercive operator is maximally monotone \cite[Corollary 20.28]{bauschke2017convex}, and since $C$ has a full domain, the operator $A+C$ is maximally monotone \cite[Corollary 25.5]{bauschke2017convex}.

We present and prove convergence for the following extended variant of FB splitting for solving \eqref{eq:monotone_inclusion}.

\begin{algorithm}[H]
	\caption{Forward--backward splitting with deviations}
	\begin{algorithmic}[1]
	    \State \textbf{Input:} initial point $x_0 \in \PrimS$, the sequences \seq{\zeta}{n}, \seq{\lambda}{n}, and \seq{\gamma}{n} as per \cref{assum:parameters}, and the metric $\norm{\cdot}_M$ with $M\in\posop{\PrimS}$.
	    \State \textbf{set:} $u_0=v_0=0$
		\For {$n=0,1,2,\ldots$}
		    \State $y_n = x_n + u_n$ \label{alg:main:y}
		    \State $z_n = x_n + \frac{\p{1 - \lambda_n} \gamma_n \beta}{2 - \lambda_n \gamma_n \beta} u_n + v_n$ \label{alg:main:z}
		    \State $p_n = (M + \gamma_n A)^{-1} \p{M z_n - \gamma_n C y_n}$ \label{alg:main:FBstep}
		    \State $x_{n+1} = x_n + \lambda_n \p{p_n - z_n}$ \label{alg:main:x}
		    \State choose $u_{n+1}$ and $v_{n+1}$ such that \label{alg:main:deviations}
		    \begin{equation}\label{eq:bound_on_deviations}
                \tfrac{\lambda_{n+1} \gamma_{n+1} \beta}{2 - \lambda_{n+1} \gamma_{n+1} \beta} \norm{u_{n+1}}_{M}^2 + \tfrac{\lambda_{n+1} \p{2 - \lambda_{n+1} \gamma_{n+1} \beta}}{4 - 2 \lambda_{n+1} - \gamma_{n+1} \beta} \norm{v_{n+1}}_{M}^2 \leq \zeta_{n} \ell_n^2
            \end{equation}
            ~~~~~is satisfied, where
            \begin{align}\label{eq:ell_n}
                \ell_{n}^2 = \tfrac{\lambda_n\p{4 - 2 \lambda_n - \gamma_n \beta}}{2} \norm{p_n - x_n + \tfrac{\lambda_n \gamma_n \beta}{2 - \lambda_n \gamma_n \beta} u_n - \tfrac{2 \p{1 - \lambda_n}}{4 -2\lambda_n-\gamma_n\beta} v_n}_M^2
            \end{align}
		\EndFor
	\end{algorithmic}
\label{alg:main}
\end{algorithm}

Our convergence analysis requires that the parameter sequences \seq{\zeta}{n}, \seq{\lambda}{n}, and \seq{\gamma}{n} satisfy the following assumption.
\begin{assumption} \label{assum:parameters}
Choose $\epsilon\in\p{0, \min\left(1, \tfrac{4}{3+\beta}\right)}$, and assume that, for all $n\in\nat$, the following hold: 

\renewcommand\theenumi\labelenumi
\renewcommand{\labelenumi}{{(\roman{enumi})}}

\begin{enumerate}[noitemsep, ref=\Cref{assum:parameters}~\theenumi]
\item $0 \leq \zeta_n \leq 1 - \epsilon$; \label{itm:assump-param-i}
\item $\epsilon \leq \gamma_n \leq \frac{4 - 3\epsilon}{\beta}$; and \label{itm:assump-param-ii} 
\item $\epsilon \leq \lambda_n \leq 2 - \frac{\gamma_n \beta}{2} - \frac{\epsilon}{2}$. \label{itm:assump-param-iii}
\end{enumerate}
\end{assumption}

The sequence \seq{\zeta}{n} relates the norm of the deviation vector $(u_{n+1},v_{n+1})$ in \eqref{eq:bound_on_deviations} to its maximum permissible value; \seq{\gamma}{n} is a sequence of step-size parameters for the FB step~\ref{alg:main:FBstep}, and \seq{\lambda}{n} can be seen as a sequence of relaxation parameters for \seq{x}{n} in step~\ref{alg:main:x} .

For our convergence analysis in~\cref{sec:conv}, we have to choose these sequences in such a way that all the coefficients multiplying the norms in~\eqref{eq:bound_on_deviations} and \eqref{eq:ell_n} have a positive lower bound. Indeed, if \seq{\gamma}{n} and \seq{\lambda}{n} satisfy \cref{assum:parameters}, then
\begin{equation}\label{eq:denominator-lower-bound-1}
  4 - 2 \lambda_n - \gamma_n \beta \geq \epsilon
\end{equation}
and
\begin{equation}\label{eq:denominator-lower-bound-2}
  2 - \lambda_n \gamma_n \beta \geq 2 - \p{2 - \frac{\gamma_n \beta}{2} - \frac{\epsilon}{2}} \gamma_n \beta = \frac{\epsilon \gamma_n \beta}{2} + 2 \p{1 - \frac{\gamma_n \beta}{2}}^2 \geq \frac{\epsilon^2 \beta}{2}.
\end{equation}

\Cref{alg:main} handles the evaluation of $C$ and $A$ in step~\ref{alg:main:FBstep} differently than the standard FB method~\eqref{eq:FB-standard} in two ways. First, the operator $M$ acts as a preconditioning for the resolvent of $A$, and secondly, the points $y_n$ and $z_n$ can be different. \Cref{alg:main} also allows for \emph{deviations} $u_n$ and $v_n$, which can be seen as design parameters of the algorithm. They can in general be chosen in a subset of $\PrimS$ with non-empty interior (if $\ell_n^2 > 0$ in step~\ref{alg:main:deviations}). Hence, the degrees of freedom in the parameter choice are determined by the dimension of $\PrimS$. It is important to note that the upper bound $\ell_n^2$, as it is seen from \eqref{eq:ell_n}, is computable at the time of selecting $u_{n+1}$ and $v_{n+1}$.  See \cite{sadeghi2022inchist} for a generalization of \Cref{alg:main}.

Below, we present some special cases of our method. We defer a more detailed discussion on special cases to \cref{sec:special_cases}.
\begin{example}\label{example:preconditioned-fb}
With the trivial choice of $u_{n+1} = v_{n+1} = 0$, the condition \eqref{eq:bound_on_deviations} is already satisfied, and \cref{alg:main} reduces to the relaxed preconditioned FB iteration
\begin{align*}
    p_n &= (M + \gamma_n A)^{-1} \p{M x_n - \gamma_n C x_n}, \\
    x_{n+1} &= x_n + \lambda_n \p{p_n - x_n}.
\end{align*}
With $M = \Id$ and $\lambda_n = 1$ ($n \in \nat$), we recover~\eqref{eq:FB-standard}.
\end{example}

\begin{example}\label{example:simplified-instance}
With $M = \Id$, $\gamma_n = \frac{1}{\beta}$, $\lambda_n = 1$, $v_n = u_n$, and $\zeta_n = 1 - \epsilon$ ($n \in \nat$), we recover the simplified version from~\eqref{eq:simplified_instance} in  \cref{sec:intro}. It is easy to see that this choice satisfies~\cref{assum:parameters}.
\end{example}

\begin{example}[no relaxation]
  With $\lambda_n = 1$ for all $n\in\nat$, \cref{alg:main} simplifies to the iteration
  \begin{align*}
    p_n &= \p{M + \gamma_n A}^{-1} \p{M \p{x_n + v_n} - \gamma_n C \p{x_n + u_n}}, \\
    x_{n+1} &= p_n - v_n
  \end{align*}
  with the norm condition
  \[
    \frac{\gamma_{n+1} \beta}{2 - \gamma_{n+1} \beta} \norm{u_{n+1}}_{M}^2 + \norm{v_{n+1}}_{M}^2 \leq \frac{\zeta_n \p{2 - \gamma_n \beta}}{2} \norm{p_n - x_n + \frac{\gamma_n \beta}{2 - \gamma_n \beta} u_n}_M^2.
  \]
\end{example}

\begin{example}[forward iteration with deviations]
  With $Ax = \set{0}$ for all $x\in\PrimS$, $v_n = 0$, and $\gamma_n = 2/\beta$ for all $n\in\nat$, \cref{alg:main} simplifies to the iteration
  \begin{align*}
    y_n &= x_n + u_n, \\
    x_{n+1} &= x_n - \tfrac{2\lambda_n}{\beta} M^{-1} C y_n
  \end{align*}
  with the norm condition
  \[
    \frac{\lambda_{n+1}}{1 - \lambda_{n+1}} \norm{u_{n+1}}_{M}^2 \leq \zeta_{n} \lambda_n \p{1-\lambda_n} \norm{\frac{1}{1-\lambda_n} u_n - \frac{2}{\beta}M^{-1} C y_n}_M^2
  \]
\end{example}

\begin{example}[backward iteration with derivations]\label{ex:backward}
  With $Cx = 0$ for all $x\in\PrimS$ and $u_n = 0$ for all $n\in\nat$, \cref{alg:main} simplifies to the iteration
  \begin{align*}
    p_n &= \p{M + \gamma_n A}^{-1} M \p{x_n + v_n}, \\
    x_{n+1} &= x_n + \lambda_n \p{p_n - x_n - v_n}.
  \end{align*}
  Since $C$ is $1/\beta$-cocoercive for all $\beta > 0$, it is possible to set $\beta = 0$ in the norm condition, which then takes the form
  \[
    \frac{\lambda_{n+1}}{2 - \lambda_{n+1}} \norm{v_{n+1}}_{M}^2 \leq \zeta_{n} \lambda_n \p{2 - \lambda_n} \norm{p_n - x_n - \frac{1 - \lambda_n}{2 - \lambda_n} v_n}_M^2.
  \]
\end{example}

\begin{remark}
Many works exist that allow for error terms in FB algorithms \cite{condat2013primal,Raguet_2013,vu2013splitting,Villa_2013}. Convergence is often based on a summability argument so that any summable sequence of errors is allowed. The strength of our condition~\eqref{eq:bound_on_deviations} is that it is iteration-wise; hence, arbitrary large errors would not be accepted. A major difference is that our algorithm does not treat the deviations as errors or inaccuracies in the computation. Instead, they are introduced to allow for actively selecting the deviations with the aim to improve performance.
\end{remark}

\section{Convergence analysis}
\label{sec:conv}

In this section, we provide a convergence analysis for \cref{alg:main}. We start by describing the points in the graph of $A+C$ constructed by \cref{alg:main} (\cref{lemma:iteration-in-graph}) and introducing a Lyapunov inequality in \cref{lemma:Lyapunov_ineq}. Both results are later used to show weak convergence in \cref{thm:main} and strong and linear convergence under a metric subregularity assumption in \cref{thm:lin_conv}. 

\begin{lemma}\label{lemma:iteration-in-graph}
  Suppose that \cref{assum:monotone_inclusion} holds. Let \seq{x}{n}, \seq{y}{n}, \seq{z}{n}, and \seq{p}{n} be sequences generated by \cref{alg:main}. Then, for all $n \in \nat$, $\p{p_n, \Delta_n} \in \gra{A + C}$, where
  \[
    \Delta_n \defeq \frac{M z_n - M p_n}{\gamma_n} - \p{C y_n - C p_n}.
  \]
  Moreover,
  \begin{align}\label{eq:Delta-estimation}
    \norm{\Delta_n}_{M^{-1}}
    &\leq \frac{1}{2 \gamma_n} \norm{2 \p{z_n - p_n} - \beta \gamma_n \p{y_n - p_n}}_M + \frac{\beta}{2} \norm{y_n - p_n}_M \nonumber \\
    &= \frac{1}{2 \gamma_n} \norm{\p{2 - \beta \gamma_n} \p{x_n - p_n} - \frac{\lambda_n \gamma_n \beta \p{2 - \gamma_n \beta}}{2 - \lambda_n \gamma_n \beta} u_n + 2 v_n}_M \nonumber \\
    &\qquad + \frac{\beta}{2} \norm{x_n - p_n + u_n}_M.
  \end{align}
\end{lemma}
  Before we prove \cref{lemma:iteration-in-graph}, note that the right-hand side of \eqref{eq:Delta-estimation} only contains data that is computed in \cref{alg:main}, whereas evaluating $\Delta_n$ requires the knowledge of $C p_n$. Therefore, \eqref{eq:Delta-estimation} can be used to check the accuracy of the current iteration or to define a stopping criterion without any extra evaluations of $C$.

  In \cref{example:simplified-instance}, \eqref{eq:Delta-estimation} reduces to
  \[
    \norm{\Delta_n} \leq \beta \norm{y_n - p_n} = \beta \norm{\p{\Id - \p{\Id + \tfrac{1}{\beta} A}^{-1} \circ \p{\Id - \tfrac{1}{\beta} C}} y_n}.
  \]
  Hence, the right-hand side of \eqref{eq:Delta-estimation} plays the role of a residual for the iteration in \cref{alg:main}.

\begin{proof}[Proof of \cref{lemma:iteration-in-graph}]
  Let $n \in \nat$. Step~\ref{alg:main:FBstep} in \cref{alg:main} is equivalent to the inclusion
  \begin{equation}
    \frac{Mz_n - Mp_n }{\gamma_n} - Cy_n \in Ap_n, 
  \end{equation}
  to which adding $C p_n$ on both sides yields the desired inclusion $\Delta_n \in \p{A + C} p_n$.
  Furthermore, we have
  \begin{multline}\label{eq:proof-of-delta-estimation}
    \norm{\Delta_n}_{M^{-1}}^2 - \p{\frac{1}{2 \gamma_n} \norm{2 \p{z_n - p_n} - \beta \gamma_n \p{y_n - p_n}}_M + \frac{\beta}{2} \norm{y_n - p_n}_M}^2 \\
    \begin{aligned}
      &= \frac{1}{\gamma_n^2} \norm{z_n - p_n}_M^2 + \norm{C y_n - C p_n}_{M^{-1}}^2 - \frac{2}{\gamma_n} \inpr{z_n - p_n}{C y_n - C p_n} \\
      &\qquad - \frac{1}{4 \gamma_n^2} \norm{2 \p{z_n - p_n} - \beta \gamma_n \p{y_n - p_n}}_M^2 - \frac{\beta^2}{4} \norm{y_n - p_n}_M^2 \\
      &\qquad - \frac{\beta}{2 \gamma_n} \norm{2 \p{z_n - p_n} - \beta \gamma_n \p{y_n - p_n}}_M \norm{y_n - p_n}_M \\
      &= \norm{C y_n - C p_n}_{M^{-1}}^2 - \frac{2}{\gamma_n} \inpr{z_n - p_n}{C y_n - C p_n} \\
      &\qquad - \frac{\beta^2}{2} \norm{y_n - p_n}_M^2 + \frac{\beta}{\gamma_n} \inpr{z_n - p_n}{y_n - p_n}_M \\
      &\qquad - \frac{\beta}{2 \gamma_n} \norm{2 \p{z_n - p_n} - \beta \gamma_n \p{y_n - p_n}}_M \norm{y_n - p_n}_M \\
      &= \norm{C y_n - C p_n}_{M^{-1}}^2 - \beta \inpr{y_n - p_n}{C y_n - C p_n} \\
      &\qquad + \frac{1}{\gamma_n} \inpr{2 \p{z_n - p_n} - \beta \gamma_n \p{y_n - p_n}}{\frac{\beta}{2} M \p{y_n - p_n} - \p{C y_n - C p_n}} \\
      &\qquad - \frac{\beta}{2 \gamma_n} \norm{2 \p{z_n - p_n} - \beta \gamma_n \p{y_n - p_n}}_M \norm{y_n - p_n}_M.
    \end{aligned}
  \end{multline}
  Notice that, by the $1/\beta$-cocoercivity of $C$ w.r.t. $\norm{\cdot}_M$,
  \begin{equation}\label{eq:cocoercivity-reformulation}
    \norm{y_n - p_n}_M \geq \frac{2}{\beta} \norm{C y_n - C p_n - \frac{\beta}{2} M \p{y_n - p_n}}_{M^{-1}}.
  \end{equation}
  The inequality part in \eqref{eq:Delta-estimation} then follows from~\eqref{eq:proof-of-delta-estimation}, using the $1/\beta$-cocoercivity again, inserting~\eqref{eq:cocoercivity-reformulation}, and applying the Cauchy--Schwarz inequality. The equality in \eqref{eq:Delta-estimation} is easily obtained by inserting the definitions of $y_n$ and $z_n$.
\end{proof}

\begin{lemma}[Lyapunov inequality] \label{lemma:Lyapunov_ineq}
Suppose that \cref{assum:monotone_inclusion,assum:parameters} hold. Let \seq{x}{n}, \seq{u}{n}, \seq{v}{n}, \seq{\ell^2}{n} be sequences generated by \cref{alg:main} and $\xs$ be an arbitrary point in $\zer{A+C}$. Then, 
\begin{equation} \label{eq:Lyapunov_0}
    \norm{x_{n+1} - \xs}_M^2 + \ell_{n}^2 \leq \norm{x_n - \xs}_M^2 + \tfrac{\lambda_n \gamma_n \beta}{2 - \lambda_n \gamma_n \beta} \norm{u_n}_M^2 + \tfrac{\lambda_n \p{2 - \lambda_n \gamma_n \beta}}{4 - 2\lambda_n - \gamma_n \beta} \norm{v_n}_M^2
\end{equation} 
and
\begin{equation}\label{eq:Lyapunov_main}
    \norm{x_{n+1} - \xs}_M^2  + \ell_{n}^2 \leq \norm{x_{n} - \xs}_M^2 + \zeta_{n-1} \ell_{n-1}^2
\end{equation}
hold for all $n\in\nat$.
\end{lemma} 

\begin{proof}
Let $n \in \nat$ be arbitrary. Step~\ref{alg:main:FBstep} in \cref{alg:main} is equivalent to the inclusion
\begin{equation}
    \frac{Mz_n - Mp_n }{\gamma_n} - Cy_n \in Ap_n. \label{eq:pn_resolvent}
\end{equation}
Since $\xs \in \zer{A+C}$, we also have
\begin{equation} \label{eq:xs_mon_inclusion}
    -C\xs \in A\xs.
\end{equation}
Using \eqref{eq:pn_resolvent}, \eqref{eq:xs_mon_inclusion}, and the monotonicity of $A$ gives
\begin{equation} \label{eq:monotonicity_A}
  0 \leq \inpr{\frac{Mz_n - Mp_n}{\gamma_n} - Cy_n + C\xs}{p_n - \xs}.
\end{equation}
By the $1/\beta$-cocoercivity of $C$ w.r.t. $\norm{\cdot}_M$ we have
\begin{equation} \label{eq:cocoercivity_B}
  \tfrac{1}{\beta}\norm{Cy_n - C\xs}_{M^{-1}}^2 \leq \inpr{Cy_n - C\xs}{y_n - \xs}.
\end{equation}
Adding \eqref{eq:monotonicity_A} and \eqref{eq:cocoercivity_B} yields
\[
    0 \leq
    \inpr{\frac{Mz_n - Mp_n}{\gamma_n} }{p_n - \xs} + \inpr{Cy_n - C\xs}{y_n - p_n} - \tfrac{1}{\beta}\norm{Cy_n-C\xs}_{M^{-1}}^2.
\]
Then, from step~\ref{alg:main:x} in \cref{alg:main}, we substitute $z_n - p_n = \frac{1}{\lambda_n} \p{x_n - x_{n+1}}$ to obtain
\begin{align*}
    0 &\leq \tfrac{1}{\gamma_n \lambda_n} \inpr{x_n - x_{n+1}}{p_n - \xs}_M + \inpr{Cy_n - C\xs}{y_n - p_n} - \tfrac{1}{\beta}\norm{Cy_n - C\xs}_{M^{-1}}^2 \\
    &= \tfrac{1}{2 \gamma_n \lambda_n} \p{\norm{x_n - \xs}_M^2 + \norm{x_{n+1} - p_n}_M^2 - \norm{x_n - p_n}_M^2 - \norm{x_{n+1} - \xs}_M^2}\\
    &\qquad\qquad + \inpr{Cy_n - C\xs}{y_n - p_n} - \tfrac{1}{\beta}\norm{Cy_n - C\xs}_{M^{-1}}^2\\
    &\leq \tfrac{1}{2 \gamma_n \lambda_n} \p{\norm{x_n - \xs}_M^2 + \norm{x_{n+1} - p_n}_M^2 - \norm{x_n - p_n}_M^2 - \norm{x_{n+1} - \xs}_M^2}\\
    &\qquad\qquad + \tfrac{\beta}{4} \norm{y_n - p_n}_M^2
\end{align*}
where we use the identity $2\inpr{a  - b}{c - d}_M = \norm{a - d}_M^2 + \norm{b - c}_M^2 - \norm{a - c}_M^2 - \norm{b - d}_M^2$ for all $a,b,c,d\in\PrimS$ and Young's inequality. Multiplying both sides of the last inequality  by $2 \gamma_n \lambda_n$ and reordering the terms yield
\begin{multline}
  \norm{x_{n+1} - \xs}_M^2 - \norm{x_n - \xs}_M^2 \\
  \begin{aligned}
  &\leq \norm{x_{n+1} - p_n}_M^2 - \norm{x_n - p_n}_M^2 + \tfrac{\lambda_n \gamma_n \beta}{2}\norm{y_n - p_n}_M^2 \\
  &=  \norm{x_n - p_n + \lambda_n \p{p_n - z_n}}_M^2 - \norm{x_n - p_n}_M^2 + \tfrac{\lambda_n \gamma_n \beta}{2} \norm{y_n - p_n}_M^2 \\
  &= \lambda_n^2 \norm{p_n - z_n}_M^2 + 2 \lambda_n \inpr{x_n - p_n}{p_n - z_n}_M + \tfrac{\lambda_n \gamma_n \beta}{2} \norm{y_n - p_n}_M^2 \\
  &= - \lambda_n \p{2 - \lambda_n} \norm{p_n - z_n}_M^2 + 2 \lambda_n \inpr{p_n - z_n}{x_n - z_n}_M\\
  &\qquad\qquad+ \tfrac{\lambda_n \gamma_n \beta}{2} \norm{y_n - p_n}_M^2,
  \end{aligned} \label{eq:Lyapunov_norms}
\end{multline}
where we, once again, used step~\ref{alg:main:x} in~\cref{alg:main} to substitute back $x_{n+1} = x_n + \lambda_n \p{p_n - z_n}$ into the expression to the right-hand side of the inequality. Now, using the definitions of $y_n$ and $z_n$ in steps~\ref{alg:main:y} and~\ref{alg:main:z} of~\cref{alg:main}, we observe that
\begin{align}
    \ell_n^2
    &= \p{\lambda_n\p{2 - \lambda_n} - \tfrac{\lambda_n \gamma_n \beta}{2}} \norm{p_n - x_n + \tfrac{\lambda_n \gamma_n \beta}{2 - \lambda_n \gamma_n \beta} u_n + \tfrac{2 \p{1 - \lambda_n}}{\gamma_n \beta - 2\p{2 - \lambda_n}} v_n}_M^2 \label{eq:Lyapunov_ell} \\
    &= \lambda_n\p{2 - \lambda_n} \norm{p_n - z_n}_M^2 + \lambda_n\p{2 - \lambda_n} \norm{\tfrac{\gamma_n \beta}{2 - \lambda_n \gamma_n \beta} u_n - \tfrac{2 - \gamma_n \beta}{\gamma_n \beta - 2\p{2 - \lambda_n}} v_n}_M^2 \nonumber \\
    &\qquad + 2 \lambda_n\p{2 - \lambda_n} \inpr{p_n - z_n}{\tfrac{\gamma_n \beta}{2 - \lambda_n \gamma_n \beta} u_n - \tfrac{2 - \gamma_n \beta}{\gamma_n \beta - 2\p{2 - \lambda_n}} v_n}_M \nonumber \\
    &\qquad - \tfrac{\lambda_n \gamma_n \beta}{2} \norm{p_n - y_n}_M^2 - \tfrac{\lambda_n \gamma_n \beta}{2} \norm{\tfrac{2}{2 - \lambda_n \gamma_n \beta} u_n + \tfrac{2 \p{1 - \lambda_n}}{\gamma_n \beta - 2\p{2 - \lambda_n}} v_n}_M^2 \nonumber \\
    &\qquad - \lambda_n \gamma_n \beta \inpr{p_n - y_n}{\tfrac{2}{2 - \lambda_n \gamma_n \beta} u_n + \tfrac{2 \p{1 - \lambda_n}}{\gamma_n \beta - 2\p{2 - \lambda_n}} v_n}_M. \nonumber
\end{align}
We can estimate the left-hand side of~\eqref{eq:Lyapunov_0} by adding \eqref{eq:Lyapunov_norms} and \eqref{eq:Lyapunov_ell}. Let us do this step by step. First, let us look at the two inner products with $p_n - z_n$.
\begin{multline*}
    2 \lambda_n \inpr{p_n - z_n}{x_n - z_n + \p{2 - \lambda_n} \p{\tfrac{\gamma_n \beta}{2 - \lambda_n \gamma_n \beta} u_n - \tfrac{2 - \gamma_n \beta}{\gamma_n \beta - 2\p{2 - \lambda_n}} v_n}}_M \\
    \begin{aligned}
        &= 2 \lambda_n \inpr{p_n - z_n}{\p{\tfrac{\gamma_n \beta \p{2 - \lambda_n}}{2 - \lambda_n \gamma_n \beta} - \tfrac{\p{1 - \lambda_n} \gamma_n \beta}{2 - \lambda_n \gamma_n \beta}} u_n - \p{1 + \tfrac{\p{2 - \lambda_n} \p{2 - \gamma_n \beta}}{\gamma_n \beta - 2\p{2 - \lambda_n}}} v_n}_M \\
        &= 2 \lambda_n \inpr{p_n - z_n}{\tfrac{\gamma_n \beta}{2 - \lambda_n \gamma_n \beta} u_n + \tfrac{\p{1 - \lambda_n} \gamma_n \beta}{\gamma_n \beta - 2\p{2 - \lambda_n}} v_n}_M
    \end{aligned}
\end{multline*}
This can be combined with the last term in~\eqref{eq:Lyapunov_ell}, so that we get
\begin{multline}\label{eq:Lyapunov_only_uv}
  \norm{x_{n+1} - \xs}_M^2 - \norm{x_n - \xs}_M^2 + \ell_n^2 \\
  \begin{aligned}
  &\leq 2 \lambda_n \gamma_n \beta \inpr{y_n - z_n}{\tfrac{1}{2 - \lambda_n \gamma_n \beta} u_n + \tfrac{\p{1 - \lambda_n}}{\gamma_n \beta - 2\p{2 - \lambda_n}} v_n}_M \\
  &\qquad + \lambda_n\p{2 - \lambda_n} \norm{\tfrac{\gamma_n \beta}{2 - \lambda_n \gamma_n \beta} u_n - \tfrac{2 - \gamma_n \beta}{\gamma_n \beta - 2\p{2 - \lambda_n}} v_n}_M^2 \\
    &\qquad - 2 \lambda_n \gamma_n \beta \norm{\tfrac{1}{2 - \lambda_n \gamma_n \beta} u_n + \tfrac{\p{1 - \lambda_n}}{\gamma_n \beta - 2\p{2 - \lambda_n}} v_n}_M^2.
  \end{aligned}
\end{multline}
With $y_n - z_n = \tfrac{2 - \gamma_n \beta}{2 - \lambda_n \gamma_n \beta} u_n - v_n$, the right-hand side of~\eqref{eq:Lyapunov_only_uv} is a quadratic expression in $u_n$ and $v_n$ alone:
\begin{multline*}
  \norm{x_{n+1} - \xs}_M^2 - \norm{x_n - \xs}_M^2 + \ell_n^2 \\
  \begin{aligned}
  &\leq 2 \lambda_n \gamma_n \beta \inpr{\tfrac{1 - \gamma_n \beta}{2 - \lambda_n \gamma_n \beta} u_n - \tfrac{\gamma_n \beta - 3 + \lambda_n}{\gamma_n \beta - 2\p{2 - \lambda_n}} v_n}{\tfrac{1}{2 - \lambda_n \gamma_n \beta} u_n + \tfrac{\p{1 - \lambda_n}}{\gamma_n \beta - 2\p{2 - \lambda_n}} v_n}_M \\
  &\qquad + \lambda_n\p{2 - \lambda_n} \norm{\tfrac{\gamma_n \beta}{2 - \lambda_n \gamma_n \beta} u_n - \tfrac{2 - \gamma_n \beta}{\gamma_n \beta - 2\p{2 - \lambda_n}} v_n}_M^2.
  \end{aligned}
\end{multline*}
In order to verify~\eqref{eq:Lyapunov_0}, it suffices to check the coefficients of $\norm{u_n}_M^2$, $\norm{v_n}_M^2$, and $\inpr{u_n}{v_n}_M$ on the right-hand side. This results in
\begin{multline*}
  \norm{x_{n+1} - \xs}_M^2 - \norm{x_n - \xs}_M^2 + \ell_n^2 \\
  \begin{aligned}
  &\leq \tfrac{2 \lambda_n \gamma_n \beta \p{1 - \gamma_n \beta} + \lambda_n \gamma_n^2 \beta^2 \p{2 - \lambda_n}}{\p{2 - \lambda_n \gamma_n \beta}^2} \norm{u_n}_M^2 \\
  &\qquad + \tfrac{- 2 \lambda_n \gamma_n \beta \p{\gamma_n \beta - 3 + \lambda_n} \p{1 - \lambda_n} + \lambda_n\p{2 - \lambda_n} \p{2 - \gamma_n \beta}^2}{\p{\gamma_n \beta - 2\p{2 - \lambda_n}}^2} \norm{v_n}_M^2 \\
  &\qquad + \tfrac{2 \lambda_n \gamma_n \beta \p{1 - \gamma_n \beta} \p{1 - \lambda_n} - 2 \lambda_n \gamma_n \beta \p{\gamma_n \beta - 3 + \lambda_n} - 2 \lambda_n \gamma_n \beta \p{2 - \lambda_n} \p{2 - \gamma_n \beta}}{\p{2 - \lambda_n \gamma_n \beta} \p{\gamma_n \beta - 2\p{2 - \lambda_n}}} \inpr{u_n}{v_n}_M \\
  &= \tfrac{\lambda_n \gamma_n \beta}{2 - \lambda_n \gamma_n \beta} \norm{u_n}_M^2 + \tfrac{\lambda_n \p{ - 2 + \lambda_n \gamma_n \beta}}{\p{\gamma_n \beta - 2\p{2 - \lambda_n}}} \norm{v_n}_M^2,
  \end{aligned}
\end{multline*}
showing~\eqref{eq:Lyapunov_0}. Finally, \eqref{eq:Lyapunov_main} follows from inserting~\eqref{eq:bound_on_deviations}.
\end{proof}

The following theorem is the main convergence result of the paper that guarantees weak convergence for the sequence of iterates obtained from \cref{alg:main}.

\renewcommand\theenumi\labelenumi
\renewcommand{\labelenumi}{{(\roman{enumi})}}

\begin{theorem}\label{thm:main}
Suppose that \cref{assum:monotone_inclusion,assum:parameters} hold. Let the sequences \seq{x}{n}, \seq{u}{n}, \seq{v}{n}, and \seq{\ell^2}{n} be generated by \cref{alg:main}. Then, the following hold:
\begin{enumerate}[noitemsep, ref=\Cref{thm:main}~\theenumi]
    \item The sequence \seq{\ell^2}{n} is summable and the sequences \seq{u}{n} and \seq{v}{n} are convergent to zero.\label{itm:thm-main-i}
    \item For all $\xs \in \zer{A+C}$, the sequence \seqsc{\norm{x_n - \xs}_M}{n} converges.\label{itm:thm-main-ii}
    \item The sequence \seq{x}{n} converges weakly to a point in $\zer{A+C}$.\label{itm:thm-main-iii}
\end{enumerate}
\end{theorem}

\begin{proof}
We start by proving~\ref{itm:thm-main-i} via a telescoping argument for~\eqref{eq:Lyapunov_main}. To this end, let $N \in \nat$. We sum~\eqref{eq:Lyapunov_main} for $n = 1, 2,\ldots,N$ to obtain
\begin{align*}
    \norm{x_{N+1} - \xs}_M^2 + \ell_{N}^2 + \sum_{n=1}^{N-1} \p{1 - \zeta_n}\ell_{n}^2  \leq  \norm{x_1 - \xs}_M^2 + \zeta_0 \ell_0^2.
\end{align*}
Then, rearranging the terms gives
\begin{align*}
    \sum_{n=1}^{N} \p{1 - \zeta_{n}}\ell_n^2 &\leq  \norm{x_1 - \xs}_M^2 - \norm{x_{N+1} - \xs}_M^2 - \zeta_N \ell_{N}^2 \\ 
    &\leq \norm{x_1-\xs}_M^2 + \zeta_0 \ell_0^2. 
\end{align*}
    Since the right hand side of the last inequality is independent of $N$, we conclude that
    \begin{align*}
        \sum_{n=0}^{\infty} \p{1 - \zeta_n} \ell_n^2 < \infty,
    \end{align*}
    which, along with $\zeta_n \leq 1 - \epsilon$ from \cref{assum:parameters}, implies that 
    \begin{equation}
        \ell_{n}^2 \to 0
    \end{equation}
    as $n \to \infty$. Then, \eqref{eq:bound_on_deviations} implies that $u_n \rightarrow 0$ and $v_n \rightarrow 0$ as $n\rightarrow\infty$. This proves~\ref{itm:thm-main-i}.

The proof of \ref{itm:thm-main-ii} follows from the property that~\eqref{eq:Lyapunov_main} defines a Lyapunov function:
since $\zeta_n \leq 1$, we get from  \eqref{eq:Lyapunov_main} that
    \begin{align*}
        \norm{x_{n+1} - \xs}_M^2 + \ell_n^2 &\leq \norm{x_n - \xs}_M^2 + \ell_{n-1}^2,
    \end{align*}
i.e., the sequence \seqsc{\norm{x_n - \xs}_M^2 + \ell_{n-1}^2}{n} is nonincreasing. As it is also nonnegative, it is convergent, say $\norm{x_n - \xs}_M^2 + \ell_{n-1}^2 \to \ell_{\xs} \geq 0$ as $n\to\infty$.
Moreover, $\ell_n^2 \to 0$ by~\ref{itm:thm-main-i} as $n\to\infty$, so $\norm{x_n - \xs}_M^2 \to \ell_{\xs}$, proving~\ref{itm:thm-main-ii}.

For the proof of \ref{itm:thm-main-iii}, recall that $\p{p_n,\Delta_n} \in \gra{A + C}$ for all $n \in \nat$ by~\cref{lemma:iteration-in-graph}. Now, by \eqref{eq:denominator-lower-bound-1}, we have $\frac{\lambda_n \p{4 - 2 \lambda_n - \gamma_n \beta}}{2} \geq \epsilon^2/2$ for all $n\in\nat$. By this and $\ell_{n} \to 0$ as $n \to \infty$, we have that
\[
    p_n - x_n + \tfrac{\lambda_n \gamma_n \beta}{2 - \lambda_n \gamma_n \beta} u_n + \tfrac{2 \p{\lambda_n - 1}}{4 - 2 \lambda_n - \gamma_n\beta} v_n \to 0.
    \]
Next, from $u_n \to 0$ and $v_n \to 0$, together with~\eqref{eq:denominator-lower-bound-1} and \eqref{eq:denominator-lower-bound-2}, we conclude that $p_n - x_n \to 0$ as $n \to \infty$. Then, by \cref{lemma:iteration-in-graph}
\begin{multline*}
    \norm{\Delta_n}_{M^{-1}}
    \leq \frac{1}{2 \gamma_n} \norm{\p{2 - \beta \gamma_n} \p{x_n - p_n} - \frac{\lambda_n \gamma_n \beta \p{2 - \gamma_n \beta}}{2 - \lambda_n \gamma_n \beta} u_n + 2 v_n}_M \\
    + \frac{\beta}{2} \norm{x_n - p_n + u_n}_M,
\end{multline*}
hence, $\Delta_n \to 0$ as $n \to \infty$.

Now, from~\ref{itm:thm-main-ii}, we know that \seqsc{\norm{x_n-\xs}_M^2}{n} is convergent, which implies that the sequence \seq{x}{n} is bounded. Therefore, the latter has at least one weakly convergent subsequence \seqsc{x_{k_n}}{n}, say $x_{k_n} \weakto \xwc \in \PrimS$ as $n \to \infty$. By the arguments above, we have $p_{k_n} \weakto \xwc$ and $\Delta_{k_n} \to 0$. Therefore, $\p{\xwc, 0} \in \gra{A + C}$ by the weak--strong closedness of $\gra{A + C}$ \cite[Proposition 20.38]{bauschke2017convex}. Then, \ref{itm:thm-main-iii} follows from \cite[Lemma 2.47]{bauschke2017convex}, and the proof is complete.
\end{proof}

\subsection{Linear convergence}

In this section, we show the linear convergence of \cref{alg:main} under the following metric subregularity assumption. 

\begin{definition}[$M$-metric subregularity]\label{def:metric_subreg}
Let $M\in\posop{\PrimS}$. A mapping $T:\PrimS\rightarrow2^{\PrimS}$ is called $M$-metrically subregular at $\bar{x}$ for $\bar{y}$ if $(\bar{x},\bar{y})\in\gra{T}$ and there exists a $\kappa\geq0$ along with neighborhoods $\mathcal{U}$ of $\bar{x}$ and $\mathcal{V}$ of $\bar{y}$ such that
\begin{equation} \label{eq:metric_subreg}
    \dist{{M}}{x}{T^{-1}(\bar{y})} \leq \kappa\dist{{M^{-1}}}{\bar{y}}{T(x)\cap\mathcal{V}}
\end{equation}
for all $x\in\mathcal{U}$.
\end{definition}
This definition is equivalent to that in \cite{dontchev2009implicit}, but uses the $M$- and $M^{-1}$-induced  norm distances instead of the standard canonical norm distance. Using this definition simplifies the notation in the linear convergence analysis. Metric subregularity is an important notion in numerical analysis. For a set-valued operator $T$ and an input vector $\bar{y}$, it simply provides an upper bound of how far a point $x$ is from being a solution to inclusion problem $\bar{y}\in T(x)$. This upper bound is given by \eqref{eq:metric_subreg} in terms of the distance of $T(x)$ from the input vector $\bar{y}$. For a detailed discussion on this subject, see \cite{dontchev2009implicit}.

\begin{theorem}[linear convergence]
\label{thm:lin_conv}
Consider the monotone inclusion problem~\eqref{eq:monotone_inclusion} and suppose that  \cref{assum:monotone_inclusion,assum:parameters} hold, that $A+C$ is $M$-metrically subregular at all $\xs\in\zer{A+C}$ for $0$, and that either $\PrimS$ is finite-dimensional or that in  \cref{def:metric_subreg} the neighborhood $\mathcal{U}$ at all $\xs\in\zer{A+C}$ is the whole space $\PrimS$. Then, there exists $0 \leq q < 1$ such that the following statements hold.
\renewcommand{\labelenumi}{(\roman{enumi})}
\begin{enumerate}[noitemsep]
\item There exists $0 < \delta < 1$ such that
  \begin{multline*}
    \distsq{M}{x_{n+1}}{\zer{A + C}} + \p{1 - \delta} \ell_n^2 \\
    \leq q \p{\distsq{M}{x_n}{\zer{A + C}} + \p{1 - \delta} \ell_{n-1}^2}
  \end{multline*}
  for all $n \geq 1$;\label{itm:thm-linear-i}
  \item there exist $x^* \in \zer{A + C}$ and $c > 0$ such that $\norm{x_n - x^*}^2 \leq cq^n$ for all $n \geq 1$. Hence, $x_n \to x^*$ even if $\PrimS$ is infinite-dimensional.\label{itm:thm-linear-ii}
\end{enumerate}
\end{theorem}

\begin{proof}
  We start by proving~\ref{itm:thm-linear-i}. Let $\xs\in\zer{A+C}$ be the weak cluster point of the sequences \seq{x}{n} and \seq{p}{n} according to~\cref{thm:main}. From the metric subregularity of $A+C$ at $x^*$ for $0$, we get $\kappa \geq 0$ and neighborhoods $\mathcal{U}$ of $x^*$ and $\mathcal{V}$ of $0$ such that 
  \begin{align}\label{eq:metric_subreg_in_proof}
    \dist{M}{x}{\zer{A+C}} &\leq \kappa\dist{{M^{-1}}}{0}{{\p{A + C} \p{x} \cap \mathcal{V}}}
  \end{align}
  for all $x \in \mathcal{U}$.

If $\PrimS$ is finite-dimensional, then $p_n \to x^*$, and there exists $n_0 \in \nat$ such that $p_n \in \mathcal{U}$ for all $n \geq n_0$. If $\PrimS$ is infinite-dimensional, then $\mathcal{U} = \PrimS$, and $p_n \in \mathcal{U}$ for all $n\in \nat$.

Now, \cref{lemma:iteration-in-graph} gives $\Delta_n \in \p{A + C} p_n$ for all $n \in \nat$, and $\Delta_n \to 0$ by the proof of \cref{thm:main}. Let $n_0 \in \nat$ be chosen such that $\Delta_n \in \mathcal{V}$ in addition to $p_n \in \mathcal{U}$ for all $n \geq n_0$.  Setting $x = p_n$ in \eqref{eq:metric_subreg_in_proof} hence gives
  \begin{multline}\label{eq:pn-delta-estimation}
    \dist{M}{p_n}{\zer{A+C}} \\
    \begin{aligned}
      &\leq \kappa\dist{{M^{-1}}}{0}{{\p{A + C} \p{p_n} \cap \mathcal{V}}} \\
      &\leq \kappa \norm{\Delta_n}_{M^{-1}} \\
      &\leq \frac{\kappa}{2 \gamma_n} \norm{\p{2 - \beta \gamma_n} \p{x_n - p_n} - \frac{\lambda_n \gamma_n \beta \p{2 - \gamma_n \beta}}{2 - \lambda_n \gamma_n \beta} u_n + 2 v_n}_M
    \end{aligned} \\
    + \frac{\beta \kappa}{2} \norm{x_n - p_n + u_n}_M
  \end{multline}
  for all $n \geq n_0$, where we used \eqref{eq:Delta-estimation} in the last step. From \cref{lemma:Lyapunov_ineq} we have that
  \begin{equation}\label{eq:Lyapunov-in-linear-convergence}
    \norm{x_{n+1}-\xs}_M^2  + \ell_{n}^2 \leq \norm{x_{n}-\xs}_M^2 + \zeta_{n-1} \ell_{n-1}^2.
  \end{equation}
  Now, set $\xs_n \defeq \Pi_{\zer{A+C}}^M\p{x_n}$. Then, from \eqref{eq:Lyapunov-in-linear-convergence}, we get
  \begin{align}\label{eq:Lyapunov3}
    \distsq{M}{x_{n+1}}{\zer{A+C}} + \ell_n^2&\leq\norm{x_{n+1}-\xs_n}_M^2 + \ell_n^2\nonumber\\
                                             &\leq \norm{x_{n}-\xs_n}_M^2 + \zeta_{n-1}\ell_{n-1}^2\nonumber\\
                                             &= \distsq{M}{x_n}{\zer{A+C}} + \zeta_{n-1}\ell_{n-1}^2.
  \end{align} 
Next, we will estimate both sides of \eqref{eq:pn-delta-estimation} in terms of $\distsq{M}{x_n}{\zer{A+C}}$, $\ell_n^2$, and $\ell_{n-1}^2$. Let $p_n^* \defeq \Pi_{\zer{A+C}}^M\p{p_n}$. Then, since $\Pi_{\zer{A + C}}$ is the projection onto a convex set w.r.t.\ the $M$-induced metric, \cite[Theorem~3.16]{bauschke2017convex} yields
  \begin{multline*}
    \distsq{M}{p_n}{\zer{A+C}} \\
    \begin{aligned}
      &\geq \norm{p_n - p_n^*}_M^2 - 2 \inpr{x_n^* - x_n}{p_n^* - x_n^*}_M \\
      &= \norm{p_n - p_n^*}_M^2 - 2 \inpr{x_n^* - x_n}{p_n^* - p_n}_M - 2 \inpr{x_n^* - x_n}{p_n - x_n^*}_M \\
      &= \norm{p_n - p_n^* - x_n^* + x_n}_M^2 - \norm{x_n^* - x_n}_M^2 - 2 \inpr{x_n^* - x_n}{p_n - x_n^*}_M \\
      &\geq \norm{x_n^* - x_n}_M^2 - 2 \inpr{x_n^* - x_n}{p_n - x_n}_M \\
      &\geq \tfrac{1}{2} \norm{x_n^* - x_n}_M^2 - 2 \norm{p_n - x_n}_M^2,
    \end{aligned}
  \end{multline*}
  where we used Young's inequality in the last step. Combining this with \eqref{eq:pn-delta-estimation} gives
  \begin{multline}\label{eq:xn-delta-estimation}
    \tfrac{1}{2} \distsq{M}{x_n}{\zer{A+C}} \\
    \begin{aligned}
      &\leq \Biggl(\frac{\kappa}{2 \gamma_n} \norm{\p{2 - \beta \gamma_n} \p{x_n - p_n} - \frac{\lambda_n \gamma_n \beta \p{2 - \gamma_n \beta}}{2 - \lambda_n \gamma_n \beta} u_n + 2 v_n}_M \\
      &\qquad + \frac{\beta \kappa}{2} \norm{x_n - p_n + u_n}_M\Biggr)^2 + 2 \norm{p_n - x_n}_M^2 \\
      &\leq \frac{\kappa^2}{2 \gamma_n^2} \norm{\p{2 - \beta \gamma_n} \p{x_n - p_n} - \frac{\lambda_n \gamma_n \beta \p{2 - \gamma_n \beta}}{2 - \lambda_n \gamma_n \beta} u_n + 2 v_n}_M^2 \\
      &\qquad + \frac{\beta^2 \kappa^2}{2} \norm{x_n - p_n + u_n}_M^2 + 2 \norm{p_n - x_n}_M^2,
    \end{aligned}
  \end{multline}
  where we used Young's inequality in the last step.
  It remains to estimate the right-hand side of \eqref{eq:xn-delta-estimation} in terms of $\ell_n^2$ and $\ell_{n-1}^2$. To this end, we use the following lemma.
  \let\qed\relax
\end{proof}

\begin{lemma}\label{lem:x-p-u-v-ell-estimation}
  Let \seq{x}{n}, \seq{p}{n}, \seq{u}{n}, \seq{v}{n}, and \seq{\ell^2}{n} be generated by \cref{alg:main} under \cref{assum:parameters}, and let \seq{\mathfrak a}{n}, \seq{\mathfrak b}{n}, and \seq{\mathfrak c}{n} be bounded sequences of real numbers. Then there exist \(c_1, c_2 > 0\) (which do not depend on $n$) such that
  \[
    \norm{\mathfrak a_n \p{p_n - x_n} + \mathfrak b_n u_n + \mathfrak c_n v_n}_M^2 \leq c_1 \ell_n^2 + c_2 \ell_{n-1}^2.
  \]
\end{lemma}
\begin{proof}
  The assertion is proven by repeatedly applying Young's inequality and subsequently using the norm condition~\eqref{eq:bound_on_deviations}:
  \begin{multline*}
    \norm{\mathfrak a_n \p{p_n - x_n} + \mathfrak b_n u_n + \mathfrak c_n v_n}_M^2 \\
    \begin{aligned}
      &= \Biggl\|\mathfrak a_n \p{p_n - x_n + \frac{\lambda_n \gamma_n \beta}{2 - \lambda_n \gamma_n \beta} u_n - \frac{2 \p{1 - \lambda_n}}{4 - 2 \lambda_n - \gamma_n \beta} v_n} \\
      &\qquad + \p{\mathfrak b_n - \frac{\lambda_n \gamma_n \beta \mathfrak a_n}{2 - \lambda_n \gamma_n \beta}} u_n + \p{\mathfrak c_n + \frac{2 \mathfrak a_n \p{1 - \lambda_n}}{4 - 2 \lambda_n - \gamma_n \beta}} v_n\Biggr\|_M^2 \\
      &\leq 2 \mathfrak a_n^2 \norm{p_n - x_n + \frac{\lambda_n \gamma_n \beta}{2 - \lambda_n \gamma_n \beta} u_n - \frac{2 \p{1 - \lambda_n}}{4 - 2 \lambda_n - \gamma_n \beta} v_n}_M^2 \\
      &\qquad + 2 \norm{\p{\mathfrak b_n - \frac{\lambda_n \gamma_n \beta \mathfrak a_n}{2 - \lambda_n \gamma_n \beta}} u_n + \p{\mathfrak c_n + \frac{2 \mathfrak a_n \p{1 - \lambda_n}}{4 - 2 \lambda_n - \gamma_n \beta}} v_n}_M^2 \\
      &\leq \frac{4 \mathfrak a_n^2}{\lambda_n \p{4 - 2 \lambda_n - \gamma_n \beta}} \ell_n^2 \\
      &\qquad + 4 \p{\mathfrak b_n - \frac{\lambda_n \gamma_n \beta \mathfrak a_n}{2 - \lambda_n \gamma_n \beta}}^2 \norm{u_n}_M^2 + 4 \p{\mathfrak c_n + \frac{2 \mathfrak a_n \p{1 - \lambda_n}}{4 - 2 \lambda_n - \gamma_n \beta}}^2 \norm{v_n}_M^2 \\
      &\leq \frac{4 \mathfrak a_n^2}{\lambda_n \p{4 - 2 \lambda_n - \gamma_n \beta}} \ell_n^2 + 4 \mathfrak d_n \zeta_{n-1} \ell_{n-1}^2
    \end{aligned}
  \end{multline*}
  with
  \begin{multline*}
    \mathfrak d_n \defeq \max\Biggl\{ \frac{2 - \lambda_n \gamma_n \beta}{\lambda_n \gamma_n \beta} \p{\mathfrak b_n - \frac{\lambda_n \gamma_n \beta \mathfrak a_n}{2 - \lambda_n \gamma_n \beta}}^2, \\
    \frac{4 - 2 \lambda_n - \gamma_n \beta}{\lambda_n \p{2 - \lambda_n \gamma_n \beta}} \p{\mathfrak c_n + \frac{2 \mathfrak a_n \p{1 - \lambda_n}}{4 - 2 \lambda_n - \gamma_n \beta}}^2 \Biggr\}.
  \end{multline*}
  It is straightforward to show, by using \cref{assum:parameters}, that $\frac{4 \mathfrak a_n^2}{\lambda_n \p{4 - 2 \lambda_n - \gamma_n \beta}}$ and $4 \mathfrak d_n \zeta_{n-1}$ are bounded, completing the proof.
\end{proof}

Now, we are in the position to complete the argument of this section's main result.
\begin{proof}[Proof of \cref{thm:lin_conv} continued]
  Since all the relevant coefficients on the right-hand side of \eqref{eq:xn-delta-estimation} are bounded due to \cref{assum:parameters}, using \cref{lem:x-p-u-v-ell-estimation} on all the norms and combining the results yields $c_1, c_2 > 0$ such that
  \[
    \tfrac{1}{2} \distsq{M}{x_n}{\zer{A+C}} \leq c_1 \ell_n^2 + c_2 \ell_{n-1}^2.
  \]
  Multiplying this with any $\delta' > 0$ and adding \eqref{eq:Lyapunov3} gives
  \begin{multline*}
    \distsq{M}{x_{n+1}}{\zer{A+C}} + \p{1 - \delta' c_1} \ell_n^2 \\
    \begin{aligned}
      &\leq \p{1 - \tfrac{\delta'}{2}} \distsq{M}{x_n}{\zer{A+C}} + \p{\zeta_{n-1} + \delta' c_2} \ell_{n-1}^2 \\
      &\leq \p{1 - \tfrac{\delta'}{2}} \distsq{M}{x_n}{\zer{A+C}} + \p{1 - \epsilon + \delta' c_2} \ell_{n-1}^2.
    \end{aligned}
  \end{multline*} 
  Choosing (for example) $\delta'$ as the smaller of the two solutions to
  \[
    \p{1 - \tfrac{\delta'}{2}} \p{1 - \delta' c_1} = \p{1 - \epsilon + \delta' c_2},
  \]
  namely
  \begin{equation}\label{eq:linear-delta}
    \delta' = \frac{1 + 2 c_1 + 2 c_2}{2 c_1} - \sqrt{\frac{\p{1 + 2 c_1 + 2 c_2}^2}{4 c_1^2} - \frac{2 \epsilon}{c_1}},
  \end{equation}
  proves \cref{itm:thm-linear-i} with $\delta = \delta' c_1$ and $q = 1 - \delta'/2$. 
  For the proof of \cref{itm:thm-linear-ii}, choose $c_1', c_2' > 0$ according to \cref{lem:x-p-u-v-ell-estimation} such that
  \begin{equation}\label{eq:xn-ell-estimation}
    \norm{x_{n+1} - x_n}_M^2 = \lambda_n^2 \norm{p_n - x_n - \frac{\p{1 - \lambda_n} \gamma_n \beta}{2 - \lambda_n \gamma_n \beta} u_n - v_n}_M^2 \leq c_1' \ell_n + c_2' \ell_{n-1}
  \end{equation}
for all $n \geq 1$. From~\cref{itm:thm-linear-i}, we get $\delta > 0$ and $0 \leq q < 1$ such that
  \begin{align*}
    \distsq{M}{x_{n+1}}{\zer{A + C}} &+ \p{1 - \delta} \ell_n^2 \leq q \p{\distsq{M}{x_n}{\zer{A + C}} + \p{1 - \delta} \ell_{n-1}^2}
  \end{align*}
for all $n \geq 1$. Repeatedly applying this relation gives
  \begin{align*}
    \ell_n^2
    &\leq \frac{1}{1 - \delta} \p{\distsq{M}{x_{n+1}}{\zer{A + C}} + \p{1 - \delta} \ell_n^2} \\
    &\leq \frac{q^n}{1 - \delta} \p{\distsq{M}{x_1}{\zer{A + C}} + \p{1 - \delta} \ell_0^2}.
  \end{align*}
  Inserting into \eqref{eq:xn-ell-estimation} and taking square roots on both sides yields
  \[
    \norm{x_{n+1} - x_n}_M \leq q^{n/2} \sqrt{\frac{c_1' + c_2'/ q}{1 - \delta} \p{\distsq{M}{x_1}{\zer{A + C}} + \p{1 - \delta} \ell_0^2}}.
  \]
 Let us choose $m > n \geq 1$ and apply the triangle inequality,
 \begin{align}
   \norm{x_m - x_n}_M
   &\leq \sum_{k = n}^{m-1} \norm{x_{k+1} - x_k}_M \nonumber \\
   &\leq \sum_{k = n}^{m-1} q^{k/2} \sqrt{\frac{c_1' + c_2'/ q}{1 - \delta} \p{\distsq{M}{x_1}{\zer{A + C}} + \p{1 - \delta} \ell_0^2}} \nonumber \\
   &\leq \sum_{k = n}^\infty q^{k/2} \sqrt{\frac{c_1' + c_2'/ q}{1 - \delta} \p{\distsq{M}{x_1}{\zer{A + C}} + \p{1 - \delta} \ell_0^2}} \nonumber \\
   &= q^{n/2} \frac{1}{1 - \sqrt q}\sqrt{\frac{c_1' + c_2'/ q}{1 - \delta} \p{\distsq{M}{x_1}{\zer{A + C}} + \p{1 - \delta} \ell_0^2}} \label{eq:linear-Cauchy}
 \end{align}
 showing that \seq{x}{n} is a Cauchy sequence, hence $x_n \to x^*$ as $n \to \infty$ with $x^*$ from~\cref{thm:main}. The other claim of~\cref{itm:thm-linear-ii} follows by letting $m \to \infty$ in \eqref{eq:linear-Cauchy}.
\end{proof}

\begin{remark}
    The analysis in~\cref{sec:conv} requires $\beta > 0$, but it can in an analogous way be done with the choice $C = 0$ and $\beta = 0$ without division by zero, leading to the iteration and safeguarding condition mentioned in \cref{ex:backward}.
\end{remark}

\section{Special cases}
\label{sec:special_cases}
In this section, we present some special cases of our algorithm.

\subsection{Primal--dual splitting with deviations} \label{subsec:primal--dual}
We are concerned with the primal inclusion problem of finding $x \in \PrimS$ such that
\begin{align} \label{eq:inclusion_CondatVu}
    0 \in Ax + L^*B(Lx) + Cx
\end{align}
under the following assumption.
\begin{assumption}
\label{assum:primal_dual_CondatVu}
We assume that
\renewcommand{\labelenumi}{\emph{(\roman{enumi})}}
\begin{enumerate}
    \item $A:\PrimS \rightarrow 2^{\PrimS}$ is a maximally monotone operator;
    \item $B:\DualS \rightarrow 2^{\DualS}$ is a maximally monotone operator;
    \item $L:\PrimS \rightarrow \DualS$ is a bounded linear operator;
    \item $C:\PrimS \rightarrow \PrimS$ is a $\tfrac{1}{\beta}$-cocoercive operator with respect to $\|\cdot\|$;
    \item the solution set $\zer{A+L^*BL+C}  \defeq \{ x\in\PrimS : 0\in Ax + L^*B(Lx) + Cx \}$ is nonempty.
\end{enumerate}
\end{assumption}

Problem~\eqref{eq:inclusion_CondatVu} can be translated to a primal--dual problem \cite{he2012convergence}: $x\in \PrimS$ is a solution to \eqref{eq:inclusion_CondatVu} if and only if there exists $\mu \in B \p{L x}$ (the \emph{dual variable}) such that
\begin{equation}
\begin{aligned}
0&\in Ax + L^*\mu + Cx,\\
0&\in -Lx + B^{-1}\mu.
\end{aligned}
\label{eq:primal-dual-inclusions}
\end{equation}
Define the primal--dual pair $w \defeq (x,\mu) \in \PrimS \times \DualS$. Then, \eqref{eq:primal-dual-inclusions} can be restated as
\begin{align} \label{eq:primal_dual_inclusion}
    0\in\mathcal{A}w + \mathcal{C}w,
\end{align}
where (with slight abuse of notation in the infinite-dimensional setting) 
\begin{align}\label{eq:new_A_C}
   \mathcal{A} = \begin{bmatrix}A & L^*\\-L & B^{-1}\end{bmatrix}, && \mathcal{C} = \begin{bmatrix}C & 0\\0 & 0\end{bmatrix}.
\end{align}
The operator $\mathcal{A}$ is maximally monotone  by \cite[Proposition 26.32]{bauschke2017convex} and $\mathcal{C}$ is  $1/\beta$-cocoercive with respect to the metric $\norm{\cdot}_M$, with
\begin{align}\label{eq:VuCondat_M}
    M = \begin{bmatrix}
        I & -\tau L^*\\
        -\tau L & \tau\sigma^{-1}I
    \end{bmatrix}   
\end{align}
where $\sigma,\tau>0$ such that $\sigma\tau\|L\|^2 < 1$.

The translation of~\eqref{eq:inclusion_CondatVu} to \eqref{eq:primal_dual_inclusion} via the two operators $\mathcal A$ and $\mathcal C$ shows that \cref{alg:main:deviations} using the metric $M$ can be used to solve problem \eqref{eq:inclusion_CondatVu}. We present this special case in \cref{alg:CondatVu_corrected}, along with the subsequent result on its convergence.
\begin{algorithm} []
	\caption{}
	\begin{algorithmic}[1]
	    \State \textbf{Input:} $(x_0,\mu_0) \in \PrimS\times\DualS$, the sequences \seq{\lambda}{n} and \seq{\zeta}{n} as defined in \cref{assum:parameters}, and $\sigma,\tau>0$ such that $\sigma\tau\|L\|^2 < 1$.
	    \State \textbf{set:} $u_{x,0}=v_{x,0}=0$, $v_{\mu,0}=0$.
		\For {$n=0,1,2,\ldots$}
		    \State $\Tilde{x}_n = x_n+ u_{x,n}$\label{alg:line:4}
		    \State $\begin{bmatrix} \hat{x}_n\\ \hat{\mu}_n \end{bmatrix} = \begin{bmatrix}x_n\\\mu_n \end{bmatrix}+ \begin{bmatrix} \tfrac{(1-\lambda_n)\tau\beta}{2-\lambda_n\tau\beta}u_{x,n}+v_{x,n}\\v_{\mu,n}\end{bmatrix}$
		    \State $\begin{bmatrix} p_{x,n}\\ p_{\mu,n} \end{bmatrix} = \begin{bmatrix} J_{\tau A} \left(\hat{x}_n-\tau L^*\hat{\mu}_n - \tau C\Tilde{x}_n \right)\\J_{\sigma B^{-1}} \left(\hat{\mu}_n+\sigma L(2p_{x,n}-\hat{x}_n)\right) \end{bmatrix}$
		    \State $\begin{bmatrix}x_{n+1}\\\mu_{n+1} \end{bmatrix} = \begin{bmatrix}x_n\\\mu_n \end{bmatrix} + \lambda_n \left( \begin{bmatrix}p_{x,n}\\p_{\mu,n}\end{bmatrix} - \begin{bmatrix}\hat{x}_n\\\hat{\mu}_n \end{bmatrix} \right)$
		    \State choose $u_{n+1}=(u_{x,n+1},u_{\mu,n+1})$ and  $v_{n+1}=(v_{x,n+1},v_{\mu,n+1})$ such that
		    \begin{multline}\label{eq:pd_deviations_bound}
            \tfrac{\lambda_{n+1}\tau\beta}{2-\lambda_{n+1}\tau\beta}\norm{u_{x,n+1}}^2 + \tfrac{\lambda_{n+1} \p{2-\lambda_{n+1}\tau\beta}}{4-2\lambda_{n+1}-\tau\beta}\norm{\begin{bmatrix}v_{x,n+1}\\v_{\mu,n+1}\end{bmatrix}}_{M}^2\\
                 \leq \zeta_{n}\tfrac{\lambda_n(4-2\lambda_n-\tau\beta)}{2}\Biggl\lVert\begin{bmatrix}p_{x,n}\\p_{\mu,n}\end{bmatrix} - \begin{bmatrix}x_{n}\\\mu_{n}\end{bmatrix} + \tfrac{\lambda_n\tau\beta}{2-\lambda_n\tau\beta}\begin{bmatrix}u_{x,n}\\0\end{bmatrix} \\
                -  \tfrac{2(1 - \lambda_n)}{4-2\lambda_n-\tau\beta}\begin{bmatrix}v_{x,n}\\v_{\mu,n}\end{bmatrix}\Biggr\rVert_{M}^2
            \end{multline}
	
		\EndFor
	\end{algorithmic}
\label{alg:CondatVu_corrected}
\end{algorithm}

\begin{corollary}\label{cor:primal_dual_split}
Consider monotone inclusions \eqref{eq:primal_dual_inclusion} and suppose that Assumption \ref{assum:primal_dual_CondatVu} holds. Let \seq{x}{n} and \seq{\mu}{n} denote the primal and the dual sequences, respectively, that are obtained from \cref{alg:CondatVu_corrected}. Then \seq{x}{n} converges weakly to a point in $\zer{A+L^*BL+C}$.
\end{corollary}
\begin{proof}
In \cref{alg:main:deviations}, replace $A$ by $\mathcal{A}$ and $C$ by $\mathcal{C}$ as devised by \eqref{eq:new_A_C}, and substitute $(x_n,\mu_n)$ in place of $x_n$, and also set $p_n=(p_{x,n},p_{\mu,n})$, $y_n=(\Tilde{x}_n,\mu_n)$, $z_n=(\hat{x}_n,\hat{\mu}_n)$, $u_n= (u_{x,n},0)$, $v_n= (v_{x,n},v_{\mu,n})$, $M$ as is in \eqref{eq:VuCondat_M}, and $\gamma_n=\tau$ ($n \in \nat$). These changes, along with the update formula 
\begin{align*}
    p_{n}=(p_{x,n},p_{\mu,n})&=(M + \tau\mathcal{A})^{-1}(M z_n - \tau\mathcal{C}y_n)\\
    &=\begin{bmatrix}I+\tau A&0\\-2\tau L&\tau\sigma^{-1}I+\tau B^{-1}\end{bmatrix}^{-1} \begin{bmatrix}\hat{x}_n -\tau L^*\hat{\mu}_n - \tau C\Tilde{x}_n\\-\tau L\hat{x}_n + \tau\sigma^{-1}\hat{\mu}_n\end{bmatrix}\\
    &=\begin{bmatrix}
        (I+\tau A)^{-1}(\hat{x}_n-\tau L^*\hat{\mu}_n - \tau C\Tilde{x}_n)\\
        (I+\sigma B^{-1})^{-1}(\hat{\mu}_n+\sigma L(2p_{x,n}-\hat{x}_n))
    \end{bmatrix}\\
    &= \begin{bmatrix} J_{\tau A} \left(\hat{x}_n-\tau L^*\hat{\mu}_n - \tau C\Tilde{x}_n \right)\\J_{\sigma B^{-1}} \left(\hat{\mu}_n+\sigma L(2p_{x,n}-\hat{x}_n)\right) \end{bmatrix},
\end{align*}
result in \cref{alg:CondatVu_corrected}. Therefore, \cref{alg:CondatVu_corrected} is a special instance of \cref{alg:main:deviations}; and the corollary is an immediate consequence of \cref{thm:main}.
\end{proof}

\begin{remark}
In \cref{alg:CondatVu_corrected}, it might be expected that we get  $\Tilde{\mu}_n = \mu_n + u_{\mu,n}$, which is the dual counterpart of  $\Tilde{x}_n = x_n + u_{x,n}$, but we do not. That is because the corresponding part of $\Tilde{\mu}_n$ of the operator $\mathcal{C}$ in \eqref{eq:new_A_C}, i.e. its second column, is zero, and thus, there is no need to define the dual counterpart of $\Tilde{x}_n$.
\end{remark}

\begin{remark}
In \cref{alg:CondatVu_corrected}, letting all deviations $u_{x, n}$, $v_{x, n}$, $v_{\mu,n}$ ($n \in \nat$) be zero and $\lambda_n=1$ give 
\begin{align*}
    x_{n+1} &= J_{\tau A} \left(x_n-\tau L^*\mu_n - \tau Cx_n \right),\\
    \mu_{n+1} &= J_{\sigma B^{-1}} \left(\mu_n+\sigma L(2x_{n+1}-x_n)\right).
\end{align*}
This is the Condat--V{\~u} algorithm in its basic form \cite{condat2013primal,vu2013splitting}, which, with $C=0$, reduces to the basic form of the Chambolle--Pock primal--dual method \cite{chambolle2011first}. 
\end{remark}

\begin{remark}
By letting $C = 0$, $\beta=0$, and $u_{x,n}=0$  for all $n\in\mathbb{N}$ in \cref{alg:CondatVu_corrected}, we arrive at a Chambolle--Pock method with deviations and the condition~\eqref{eq:pd_deviations_bound} reduces to
\begin{equation*}
    \begin{aligned}
    \norm{\begin{bmatrix}v_{x,n+1}\\v_{\mu,n+1}\end{bmatrix}}_{M}^2\leq \zeta_{n}\tfrac{(2-\lambda_{n+1})(2-\lambda_n)\lambda_n}{\lambda_{n+1}} \norm{\begin{bmatrix}p_{x,n}\\p_{\mu,n}\end{bmatrix} - \begin{bmatrix}x_{n}\\\mu_{n}\end{bmatrix} - \tfrac{1-\lambda_{n}}{2-\lambda_{n}}\begin{bmatrix}v_{x,n}\\v_{\mu,n}\end{bmatrix}}_{M}^2. \quad
\end{aligned}
\end{equation*}
\end{remark}

\subsection{\Krasnoselsky--Mann iteration with deviations}
Consider the fixed-point problem 
\begin{equation}\label{eq:fixed-point-problem}
    x = Tx,
\end{equation}
where $T:\PrimS\rightarrow \PrimS$ is a nonexpansive operator. Then, by \cite[Remark~4.34, Corollary~23.9]{bauschke2017convex}, there is a maximally  monotone operator $A:\PrimS\rightarrow2^\PrimS$ for which $J_{\gamma A}= \tfrac{1}{2}\Id+\tfrac{1}{2}T$, with $\gamma>0$. This correspondence suggests that \cref{alg:main:deviations} can be used to solve \eqref{eq:fixed-point-problem}. Letting $C=0$, $\beta=0$, $M=\Id$, and $u_n=0$ for all $n\in\nat$ in \cref{alg:main:deviations}, results in \cref{alg:KM_corrected}, that can be used to solve problem~\eqref{eq:fixed-point-problem}. Weak convergence of \cref{alg:KM_corrected} is shown in \cref{cor:KM}.

\begin{corollary}\label{cor:KM}
Consider the fixed-point problem \eqref{eq:fixed-point-problem}; suppose that its solution set is nonempty and let $J_{\gamma A}=\tfrac{1}{2}\Id+\tfrac{1}{2}T$. Then, the sequence \seq{x}{n}, that is generated by \cref{alg:KM_corrected}, converges weakly to a point in the solution set of the problem.
\end{corollary}

\begin{algorithm}[]
	\caption{}
	\begin{algorithmic}[1]
	    \State \textbf{Input:} $x_0\in\PrimS$, and the sequences \seq{\lambda}{n}, \seq{\gamma}{n}, and \seq{\zeta}{n} according to \cref{assum:parameters}.
	    \State \textbf{set:} $v_0 = 0$
		\For {$n=0,1,\ldots$}
		    \State $z_n = x_n + v_n$
		    \State $p_n =\tfrac{1}{2}(\Id + T)(x_n + v_n)$
		    \State $x_{n+1} = (1-\lambda_n)x_n + \lambda_n(p_n - v_n)$
		    \State choose $v_{n+1}$ such that
		    \begin{equation} \label{eq:deviation_bound_KM}
            \begin{aligned}
                \norm{v_{n+1}}^2 \leq \zeta_{n}\tfrac{\lambda_n(2-\lambda_n)(2-\lambda_{n+1})}{\lambda_{n+1}}\norm{p_{n}-x_{n}+\tfrac{\lambda_n-1}{2-\lambda_n}v_{n}}^2
            \end{aligned}
            \end{equation}
		\EndFor
	\end{algorithmic}
\label{alg:KM_corrected}
\end{algorithm}

Setting $v_n=0$ for all $n\in\nat$ in \cref{alg:KM_corrected} results in
\begin{align*} 
    x_{n+1} =(1-\tfrac{\lambda_n}{2})x_n+\tfrac{\lambda_n}{2}T(x_n) \label{eq:KM2},
\end{align*}
which is the standard  \Krasnoselsky--Mann iteration \cite[Corollary~5.17]{bauschke2017convex}.

\section{A novel inertial primal--dual splitting algorithm}
\label{sec:novel_inertial_pd_alg}
In this section, we present a novel inertial primal--dual method to solve problem \eqref{eq:inclusion_CondatVu} with $C=0$. We construct this algorithm from \cref{alg:CondatVu_corrected} by considering a special structure for the deviation vector. We preset the deviation vector direction at the $n$-th iteration to be aligned with the momentum direction, i.e.,  $v_n =  a_n(x_n-x_{n-1},\mu_n-\mu_{n-1})$, and use the bound on the norm of deviations to compute $a_n$. Since this algorithm is an instance of \cref{alg:CondatVu_corrected}, its convergence is guaranteed by \cref{cor:primal_dual_split}.

\begin{algorithm} []
	\caption{}
	\begin{algorithmic}[1]
	    \State \textbf{Input:} $(x_0,\mu_0) \in \PrimS\times\DualS$, and the sequences \seq{\lambda}{n} and \seq{\zeta}{n} as stated in \cref{assum:parameters}.
	    \State \textbf{set:} $a_0=0$
		\For {$n=0,1,2,\ldots$}
		    \State $\begin{bmatrix}\hat{x}_n\\\hat{\mu}_n\end{bmatrix}=\begin{bmatrix}x_{n}\\\mu_{n}\end{bmatrix}+a_n\begin{bmatrix}x_n-x_{n-1}\\ \mu_n-\mu_{n-1}\end{bmatrix}$
		    \State $\begin{bmatrix} p_{x,n}\\ p_{\mu,n} \end{bmatrix} = \begin{bmatrix} J_{\tau A} \left(\hat{x}_n-\tau L^*\hat{\mu}_n \right)\\J_{\sigma B^{-1}} \left(\hat{\mu}_n+\sigma L(2p_{x,n}-\hat{x}_n)\right) \end{bmatrix}$
		    \State $\begin{bmatrix}x_{n+1}\\\mu_{n+1} \end{bmatrix} = \begin{bmatrix}x_n\\\mu_n \end{bmatrix} + \lambda_n \left( \begin{bmatrix}p_{x,n}\\p_{\mu,n}\end{bmatrix} - \begin{bmatrix}\hat{x}_n\\\hat{\mu}_n \end{bmatrix} \right)$
		    \State choose $a_{n+1}$  such that \label{alg:line:a_n}
		    \begin{equation}\label{eq:pd_deviations_bound_inertial}
            \begin{split}
                & a_{n+1}^2\norm{\begin{bmatrix}x_{n+1}-x_{n}\\\mu_{n+1}-\mu_{n}\end{bmatrix}}_{M}^2\\
                &\hspace{10mm} \leq \zeta_{n}\tfrac{\lambda_n(2-\lambda_n)(2-\lambda_{n+1})}{\lambda_{n+1}}\Biggl\lVert\begin{bmatrix}p_{x,n}-x_n\\p_{\mu,n}-\mu_n\end{bmatrix} + \tfrac{\lambda_n-1}{2-\lambda_n}a_n\begin{bmatrix}x_{n}-x_{n-1}\\\mu_{n}-\mu_{n-1}\end{bmatrix}\Biggr\rVert_{M}^2
            \end{split}
            \end{equation}
		\EndFor
	\end{algorithmic}
\label{alg:inertial_pd}
\end{algorithm}

\begin{remark}
Even though \cref{alg:inertial_pd} has similarities with translations of the algorithms of \cite{Alvarez_2000,Alvarez_2001,attouch2019convergence,Cholamjiak_2018,lorenz2015inertial} to a primal--dual framework, to the best of our knowledge, the former and the latter cannot be derived from each other, and thus, are essentially different.
\end{remark}

\subsection{Efficient evaluation of the norm condition}

In order to compute the bound on the coefficients $a_n$ using \eqref{eq:pd_deviations_bound_inertial}, one needs to compute some $M$-induced norms, which involves evaluating $L$ and $L^*$. Depending on the complexity of evaluating $L$ and $L^*$, these evaluations may be computationally expensive. However, by scrutinizing \cref{alg:inertial_pd}, it is observed that some of the previous evaluations can be reused to keep the additional computational cost low compared to the standard Chambolle--Pock algorithm. In what follows, we provide more details on how to compute the required scaled norm of the vector quantities in a computationally efficient manner.

As seen in line \ref{alg:line:a_n} of \cref{alg:inertial_pd}, at each iteration one of each $L$ and $L^*$ evaluations are performed.  Similar operations take place at each iteration of, e.g., the Chambolle--Pock algorithm. However, in our algorithm, we have other operations involving evaluations of $L$ and $L^*$. Those are due to verification of the norm condition in line 8 of \cref{alg:inertial_pd}. More specifically, since the kernel $M$ is given by \eqref{eq:VuCondat_M}
for each evaluation of $\norm{\cdot}_M$, we have one more evaluation each of $L$ and $L^*$. This can lead to a substantially higher computational cost. However, except for the first iteration, the extra $L$ and $L^*$ evaluations can be computed from the computations which are already available from previous iterations. 
That is possible due to the relations
\begin{equation}\label{eq:recursion}
    \begin{aligned}
        L\hat{x}_n &= Lx_n + b_n(Lx_n-Lx_{n-1}),\\
        L^*\hat{\mu}_n &= L^*\mu_n + b_n(L^*\mu_n-L^*\mu_{n-1}),\\
        Lx_{n+1} &= Lx_{n} + \lambda_{n}(Lp_{x,n}-L\hat{x}_n),\\
        L^*\mu_{n+1} &= L^*\mu_{n} + \lambda_{n}(L^*p_{\mu,n}-L^*\hat{\mu}_n),
    \end{aligned}
\end{equation}
which are derived from lines 5 and 7 of \cref{alg:inertial_pd}. In the relations above, for $n>0$, all quantities  to the right hand side are already computed and can be reused, except for $Lp_{x,n}$ and $L^*p_{\mu,n}$ that need to be computed via direct evaluation.

\Cref{tab:recursive_computation} provides the list of evaluations involving $L$ and $L^*$ that we need to perform at the first three iterations. It reveals that at the first iteration, we need to perform six different evaluations involving $L$ or $L^*$, of which four might be computationally heavy and two can be done cheaply. After that, i.e. for $n>0$, we only need to perform two such heavy evaluations per iteration; namely, $Lp_{x,n}$ and $L^*p_{\mu,n}$. The rest of the $L$ and $L^*$ evaluations can be done efficiently by exploiting previously computed quantities and \eqref{eq:recursion}. This keeps the computational per-iteration cost of our algorithm basically the same as that of the Chambolle--Pock algorithm.

\begin{table}[H]
  \begin{center}
 \begin{tabular}{ | c | c | c | } 
 \hline
 $n $ &  Expensive evaluations & Cheap evaluations\\
 \hline
  $0$  & $Lx_0$, $L^*\mu_0$, $Lp_{x,0}$, $L^*p_{\mu,0}$ & $Lx_1$, $L^*\mu_1$   \\
 \hline
 $1$  & $Lp_{x,1}$, $L^*p_{\mu,1}$ & $L\hat{x}_2$, $Lx_2$, $L^*\hat{\mu}_2$, $L^*\mu_2$\\
 \hline
 $2$  & $Lp_{x,2}$, $L^*p_{\mu,2}$ & $L\hat{x}_3$, $Lx_3$, $L^*\hat{\mu}_3$, $L^*\mu_3$\\
 \hline
\end{tabular}
\caption{\label{tab:recursive_computation}List of evaluations that involve $L$ and $L^*$ for the first three iterations. The second column shows direct and potentially expensive evaluations and the third column shows evaluations that can be done cheaply via the relations in \eqref{eq:recursion}.}
\end{center}
\end{table}

\section{Numerical Experiments}
\label{sec:numerical_experiments}

We solve an $l_1$-norm regularized SVM problem for classification of the form
\begin{align}\label{eq:l1-regularized-svm}
    \underset{ x }{\text{minimize }} f(Lx)+g(x),
\end{align}
given a labeled training data set $\{\theta_i,\phi_i \}_{i=1}^N$, where $\theta_i\in\reals^d$ and $\phi_i\in\{-1,1\}$ are training data and labels,  respectively, and with
\begin{align*}
    f(Lx) = \mathbf{1}^T \max\lp\mathbf{0}, \mathbf{1}-Lx\rp, && g(x) = \xi\|\omega\|_1, && L = \begin{bmatrix} \phi_1\theta_1^T & \phi_1\\ \vdots & \vdots \\ \phi_N\theta_N^T & \phi_N\end{bmatrix},
\end{align*}
where $\mathbf{0} = \p{0, \ldots, 0}^T$, $\mathbf{1} = \p{1, \ldots, 1}^T$, $x = (\omega,b)$ is the decision variable with  $b\in\reals$ and $\omega\in\reals^d$,  $\max(\cdot,\cdot)$ acts element-wise, and $\xi\geq0$ is the regularization parameter.

A point $\xs$ is a solution to \eqref{eq:l1-regularized-svm} if and only if it satisfies
\begin{equation*}
    0\in L^*\partial f(L\xs) + \partial g(\xs).
\end{equation*}
This holds, since $f$ and $g$ are proper, closed, and convex functions with full domains, and thus, $\partial f$ and $\partial g$ are maximally monotone and $L$ is a linear operator \cite[Proposition~16.42]{bauschke2017convex}. This monotone inclusion problem is an instance of \eqref{eq:inclusion_CondatVu} with $A = \partial g$, $B=\partial f$, and $C=0$. As in \cref{subsec:primal--dual}, we transform the problem into a primal--dual problem and solve it with primal--dual algorithms.

We compare our inertial primal--dual method, \cref{alg:inertial_pd}, to the standard Chambolle--Pock (CP) \cite{chambolle2011first}, and to the inertial primal--dual algorithm of Lorenz--Pock (LP) \cite{lorenz2015inertial}. 
In all experiments, we set the primal and the dual step-sizes to $\tau=\sigma = 0.99/\norm{L}$, the regularization parameter of problem \eqref{eq:l1-regularized-svm} to $\xi = 0.1$, and $\zeta_n$ is, for each $n \in \nat$, sampled  from a uniform distribution on $[0,1-10^{-6}]$. The experiments are done using the \emph{liver disorders} data-set \cite{chang2011libsvm} which has 145 samples and 5 features. The solution $(\xs,\mu^\star)$ is found by running the standard Chambolle--Pock algorithm until the residual gets smaller than $10^{-15}$.

\begin{figure}[ht]
    \centering
    \begin{subfigure}{0.49\textwidth}
    \centering
    \begin{tikzpicture}[scale=1.0]
      \node[inner sep=0pt] (fig) at (0,0) {\includegraphics[ width = 0.99\linewidth,keepaspectratio]{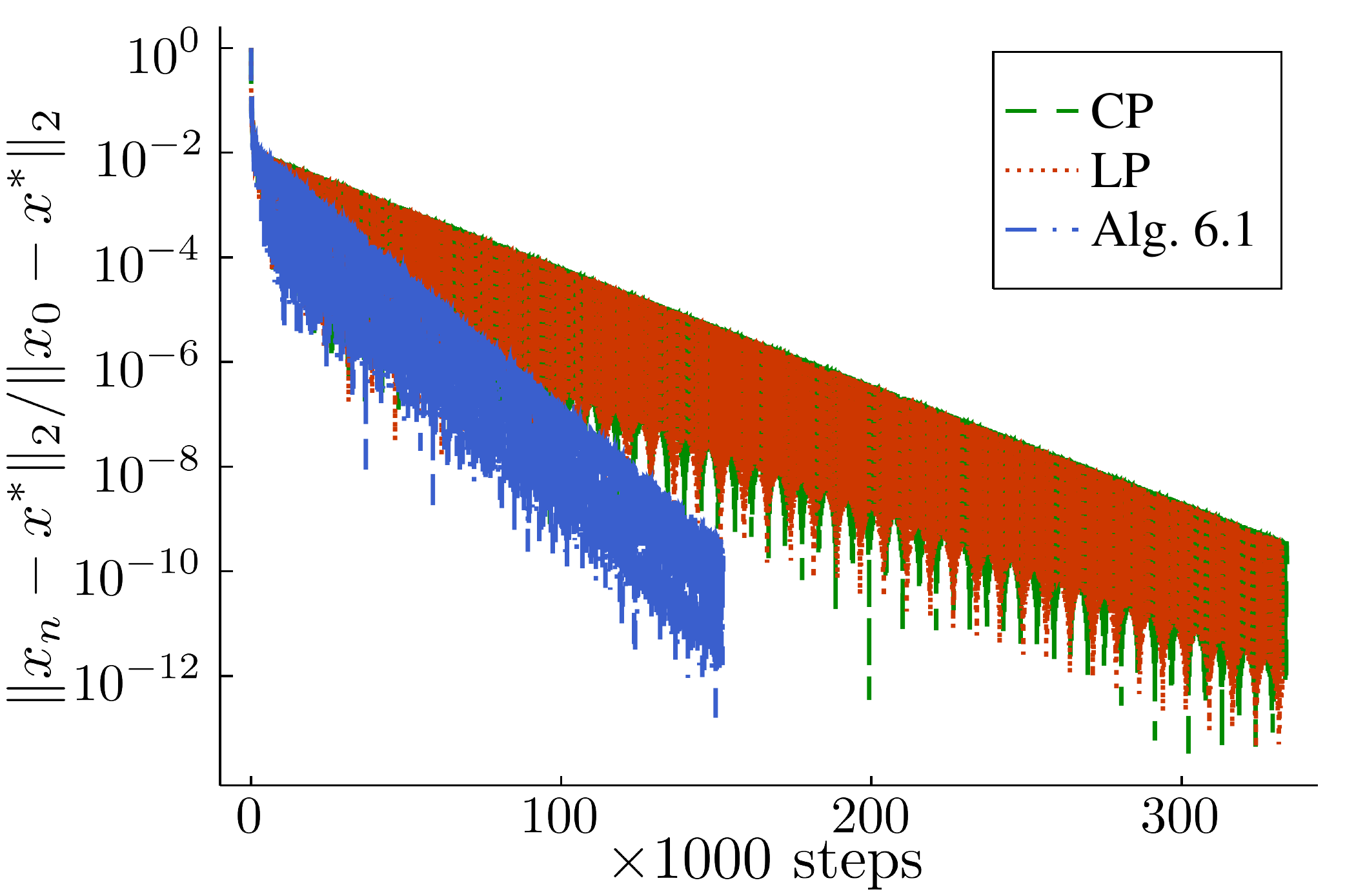}};
        \draw [draw=none, fill=white] (1.79,0.8) rectangle (2.6,1.7);
        \node[text=black,anchor=west](legend) at (1.82,1.51){{\tiny CP}};
        \node[text=black,anchor=west](legend) at (1.82,1.25){{\tiny LP}};
        \node[text=black,anchor=west](legend) at (1.82,0.97){{\tiny Alg. 4}};
        \draw [draw=none,fill=white] (-1.5,-2.2) rectangle (1.5,-1.7);
        \node[text=black](xlabel) at (0.05\linewidth,-.33\linewidth) {{\footnotesize $\times 10^3$ iteration}};
        \draw [draw=none,fill=white] (-3,-1.5) rectangle (-2.6,1.5);
        \node[rotate=90, text=black] (ylabel) at (-.48\linewidth,0.1) {{\footnotesize ${\norm{x_n-\xs}}/{\norm{x_0-\xs}}$}};
    \end{tikzpicture}
    \end{subfigure}
    \begin{subfigure}{0.49\textwidth}
    \centering
    \begin{tikzpicture}[scale=1.0]
        \node[inner sep=0pt] (fig) at (0,0) {\includegraphics[ width = 0.99\linewidth ,keepaspectratio]{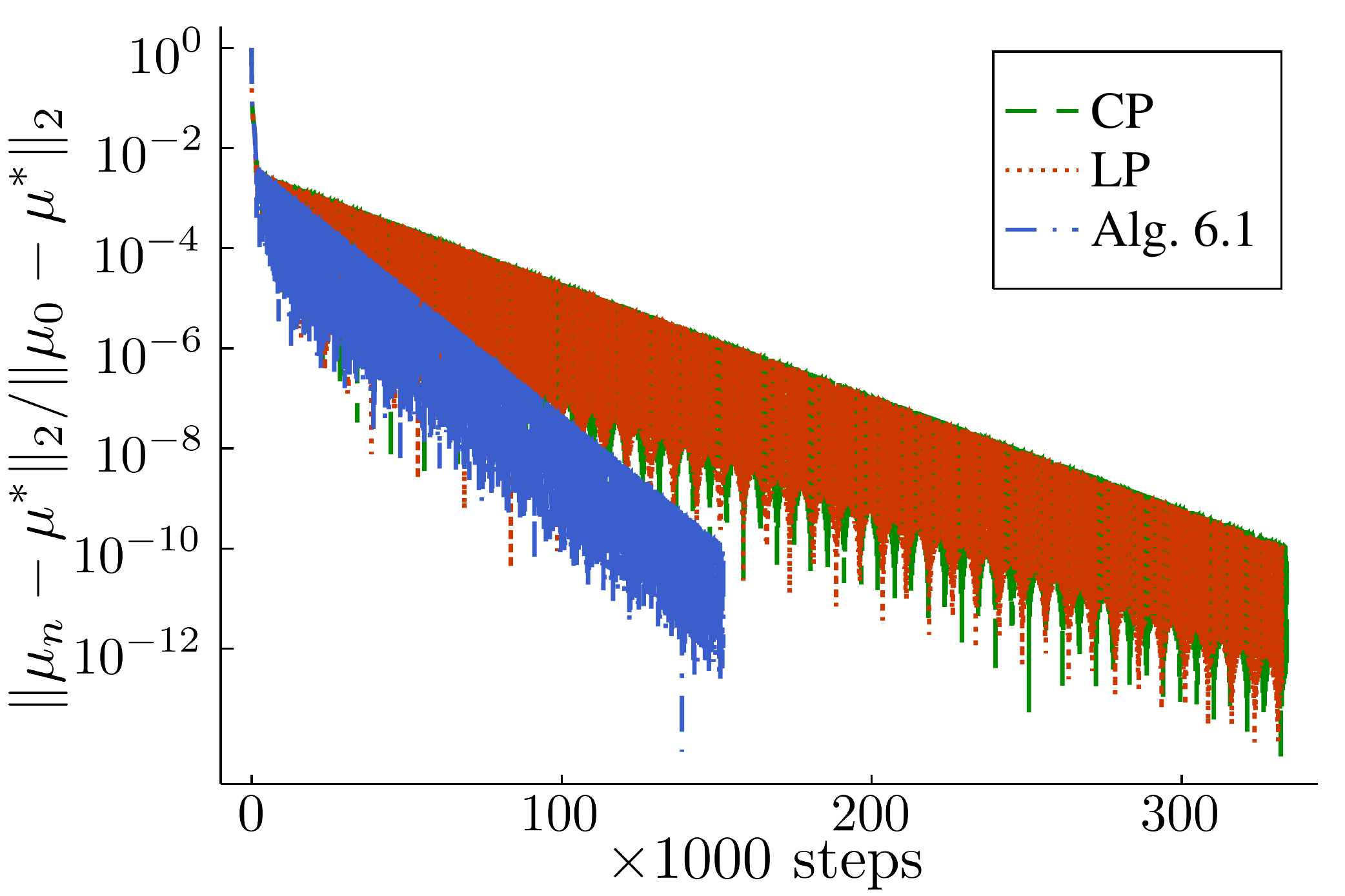}};
        \draw [draw=none, fill=white] (1.79,0.8) rectangle (2.6,1.7);
        \node[text=black,anchor=west](legend) at (1.82,1.51){{\tiny CP}};
        \node[text=black,anchor=west](legend) at (1.82,1.25){{\tiny LP}};
        \node[text=black,anchor=west](legend) at (1.82,0.97){{\tiny Alg. 4}};
        \draw [draw=none,fill=white] (-1.5,-2.2) rectangle (1.5,-1.7);
        \node[text=black](xlabel) at (0.05\linewidth,-.33\linewidth) {{\footnotesize $\times 10^3$ iteration}};
        \draw [draw=none,fill=white] (-3,-1.5) rectangle (-2.6,1.5);
        \node[rotate=90, text=black] (ylabel) at (-.48\linewidth,0.1) {{\footnotesize ${\norm{\mu_n-\mu^\star}}/{\norm{\mu_0-\mu^\star}}$}};
    \end{tikzpicture}
    \end{subfigure}
\caption{Distance to the solution vs. iteration number for the $l_1$-norm regularized SVM \eqref{eq:l1-regularized-svm} with $\xi = 0.1$, on the \emph{liver disorders} data-set \cite{chang2011libsvm} with 145 samples and 5 features. Solved using Chambolle--Pock primal--dual algorithm (CP), Lorenz--Pock inertial primal--dual method (LP), and \cref{alg:inertial_pd} with $\lambda = 1.0$. The primal and dual step-sizes are set to $\tau=\sigma = 0.99/\norm{L}$ for all algorithms. }
\label{fig:comp_alg}
\end{figure}
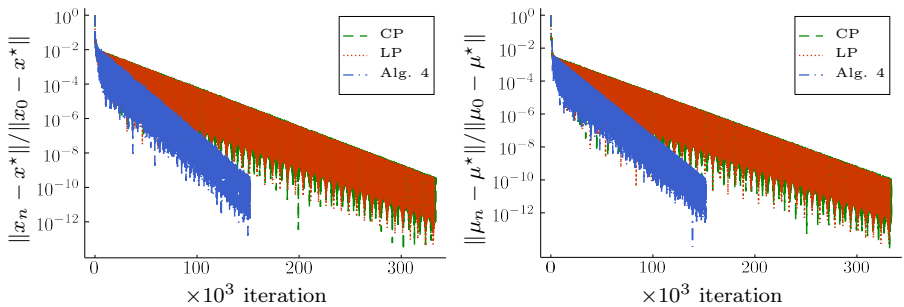

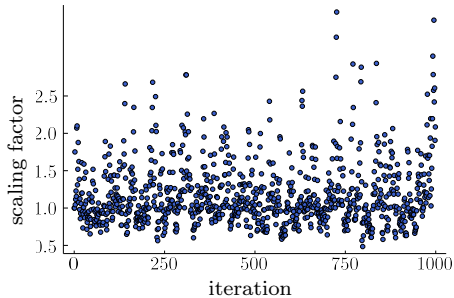
\begin{figure}[ht]
    \centering
    \begin{tikzpicture}[scale=1.0]
      \node[inner sep=0pt] (fig) at (0,0) {\includegraphics[ width = 0.49\linewidth,keepaspectratio]{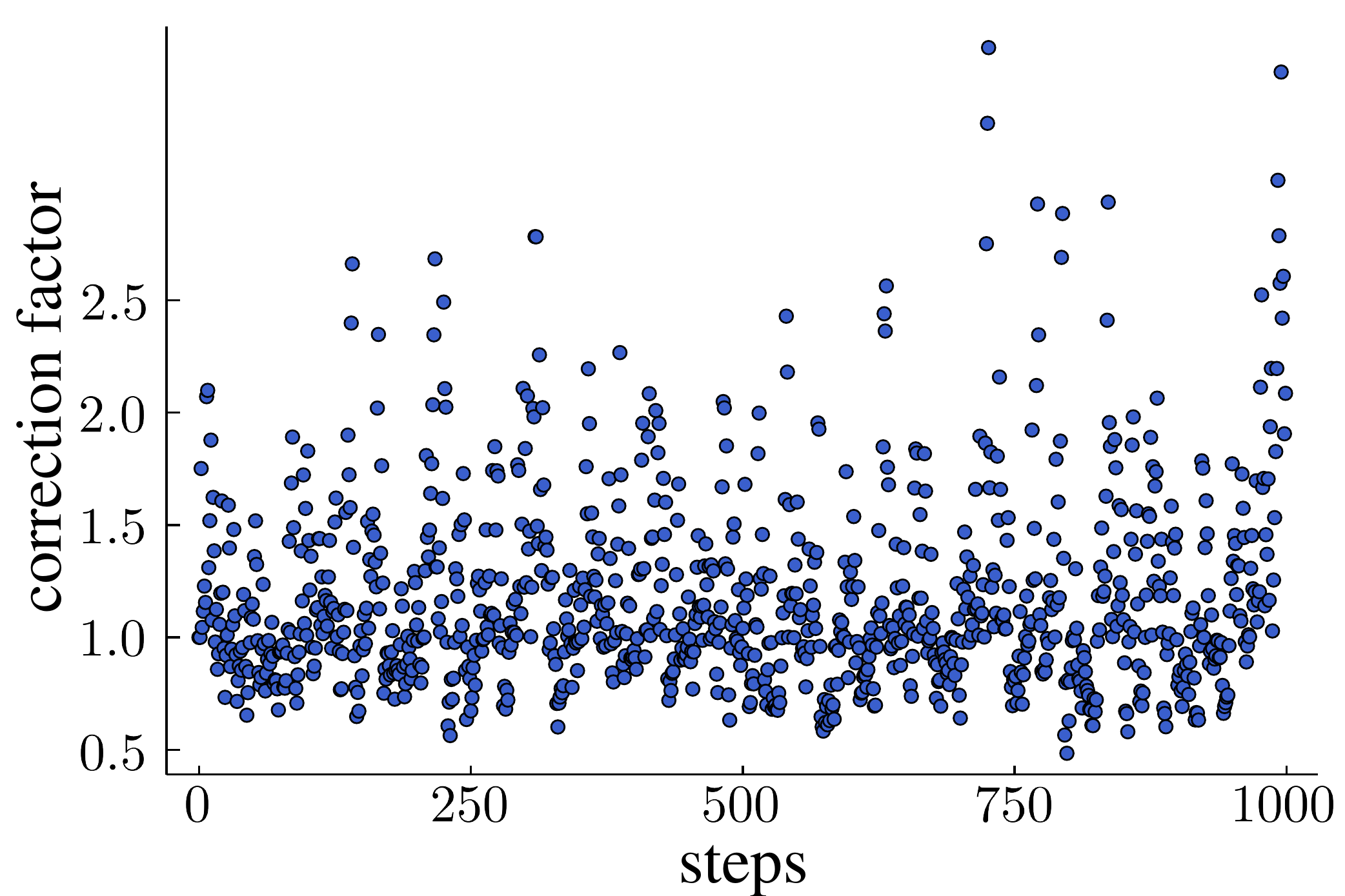}};
        \draw [draw=none,fill=white] (-1.5,-2.2) rectangle (1.5,-1.7);
        \node[text=black](xlabel) at (0.25,-1.9) {{\footnotesize iteration}};
        \draw [draw=none,fill=white] (-3,-1.5) rectangle (-2.6,1.5);
        \node[rotate=90, text=black] (ylabel) at (-2.85,-0.2) {{\footnotesize scaling factor}};
    \end{tikzpicture}
\caption{Scaling factor $a_n$ of \cref{alg:inertial_pd} in the experiment shown in \cref{fig:comp_alg} vs. iteration number for the first 1000 iterations.}
\label{fig:corr_fact}
\end{figure}

\begin{figure}[ht]
    \centering
    \begin{subfigure}{0.49\textwidth}
    \centering
    \begin{tikzpicture}[scale=1.0]
      \node[inner sep=0pt] (fig) at (0,0) {\includegraphics[width=0.99\linewidth,keepaspectratio]{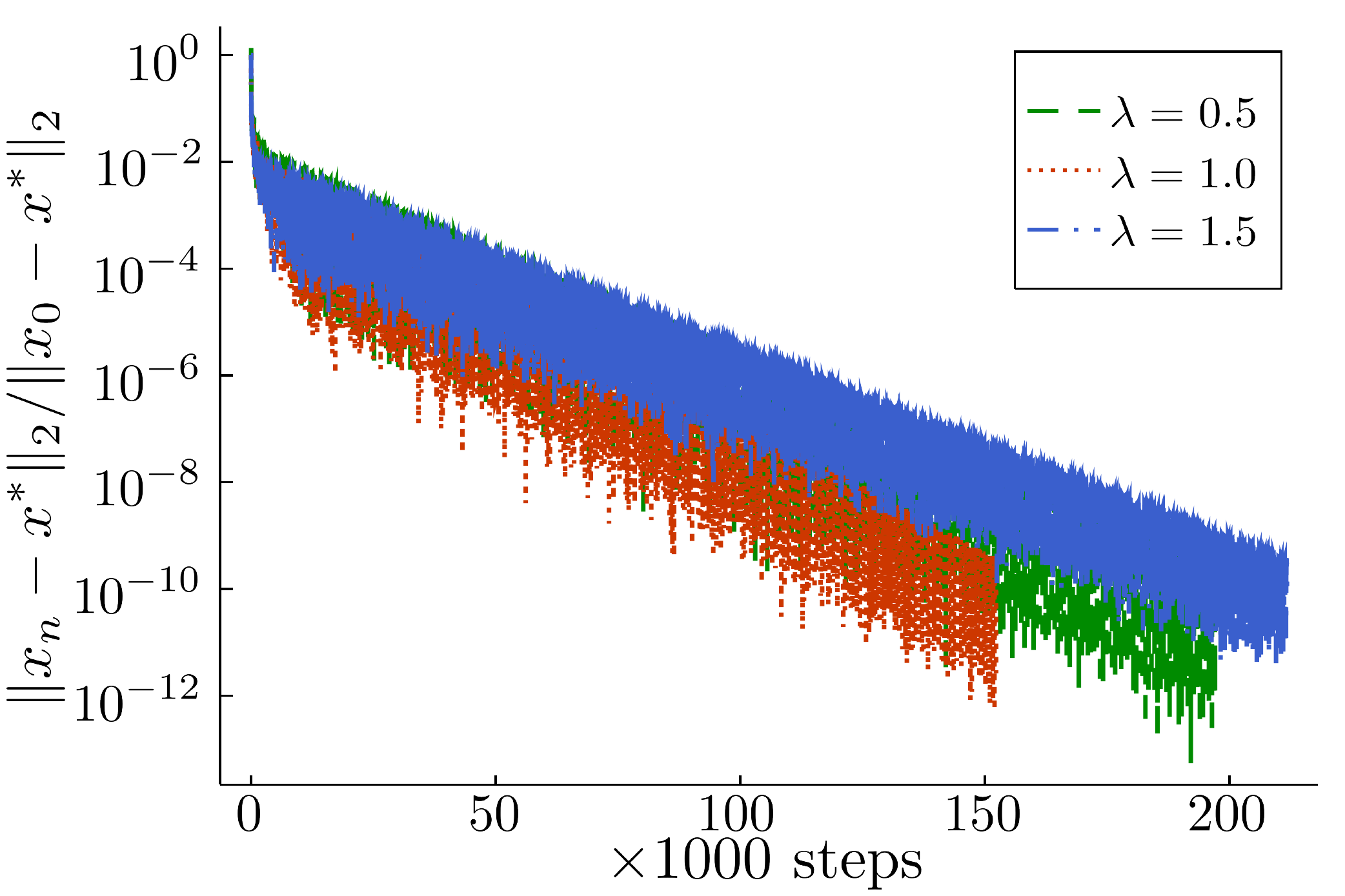}};
        \draw [draw=none, fill=white] (1.79,0.8) rectangle (2.6,1.7);
        \node[text=black,anchor=west](legend) at (1.7,1.51){{\tiny $\lambda=0.5$}};
        \node[text=black,anchor=west](legend) at (1.7,1.25){{\tiny $\lambda=1.0$}};
        \node[text=black,anchor=west](legend) at (1.7,0.97){{\tiny $\lambda=1.5$}};
        \draw [draw=none,fill=white] (-1.5,-2.2) rectangle (1.5,-1.7);
        \node[text=black](xlabel) at (0.05\linewidth,-.33\linewidth) {{\footnotesize $\times 10^3$ iteration}};
        \draw [draw=none,fill=white] (-3,-1.5) rectangle (-2.6,1.5);
        \node[rotate=90, text=black] (ylabel) at (-.48\linewidth,0.1) {{\footnotesize ${\norm{x_n-\xs}}/{\norm{x_0-\xs}}$}};
    \end{tikzpicture}
    \end{subfigure}
    \begin{subfigure}{0.49\textwidth}
    \centering
    \begin{tikzpicture}[scale=1.0]
        \node[inner sep=0pt] (fig) at (0,0) {\includegraphics[width=0.99\linewidth,keepaspectratio]{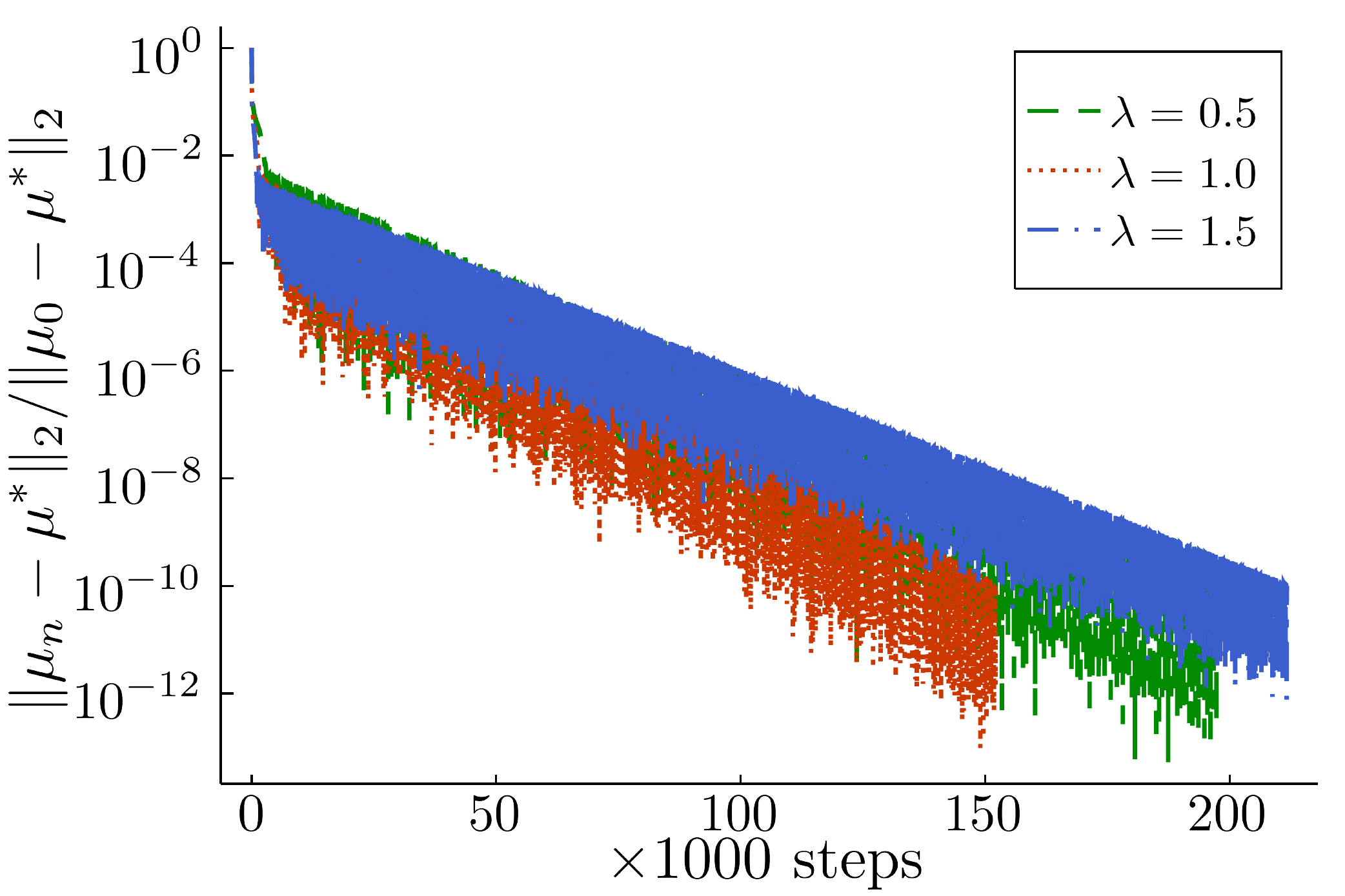}};
        \draw [draw=none, fill=white] (1.79,0.8) rectangle (2.6,1.7);
        \node[text=black,anchor=west](legend) at (1.7,1.51){{\tiny $\lambda=0.5$}};
        \node[text=black,anchor=west](legend) at (1.7,1.25){{\tiny $\lambda=1.0$}};
        \node[text=black,anchor=west](legend) at (1.7,0.97){{\tiny $\lambda=1.5$}};
        \draw [draw=none,fill=white] (-1.5,-2.2) rectangle (1.5,-1.7);
        \node[text=black](xlabel) at (0.05\linewidth,-.33\linewidth) {{\footnotesize $\times 10^3$ iteration}};
        \draw [draw=none,fill=white] (-3,-1.5) rectangle (-2.6,1.5);
        \node[rotate=90, text=black] (ylabel) at (-.48\linewidth,0.1) {{\footnotesize ${\norm{\mu_n-\mu^\star}}/{\norm{\mu_0-\mu^\star}}$}};
    \end{tikzpicture}
    \end{subfigure}
\caption{Distance to the solution vs. iteration number for the $l_1$-norm regularized SVM \eqref{eq:l1-regularized-svm} with $\xi = 0.1$, on the \emph{liver disorders} data-set \cite{chang2011libsvm} with 145 samples and 5 features. Solved using \cref{alg:inertial_pd} for some values of  $\lambda$ with $\tau=\sigma = 0.99/\norm{L}$. }
\label{fig:comp_lambda}
\end{figure}
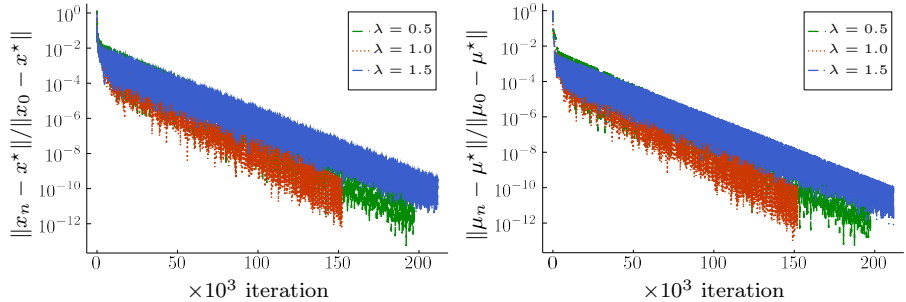

For the $l_1$-norm regularized SVM problem, since $f$ and $g$ are piece-wise linear, the resulting (primal--dual) monotone operator 
\begin{align*}
   \mathcal{A} = \begin{bmatrix}\partial g & L^*\\-L & \partial f^*\end{bmatrix}
\end{align*}
is metrically subregular at any point in the solution set of the problem for $0$, see \cite[Lemma IV.4]{Latafat2019ANR}. It therefore follows from \cref{thm:lin_conv} that the algorithm exhibits local linear convergence, see \cref{fig:comp_alg} and \cref{fig:comp_lambda}.
The figures reveal that our method needs about half the number of iterations to reach the same accuracy as the other two methods.
This improvement comes at essentially no extra computational cost.  

\Cref{fig:corr_fact} shows the first one thousand scaling factors $a_n$ of \cref{alg:inertial_pd} for the same implementation as in \cref{fig:comp_alg}. It is seen that the scaling factor attains mostly values close to one. 

In \cref{fig:comp_lambda}, the impact of the relaxation parameter $\lambda$ is investigated. In the sense of convergence rate, it interestingly seems that $\lambda = 1.0$ yields the best performance in this example.

\paragraph{Acknowledgement.} The authors would like to thank Bo Bernhardsson for his valuable feedback on this work. This research was partially supported by Wallenberg AI, Autonomous Systems and Software Program (WASP) funded by the Knut and Alice Wallenberg Foundation. Sebastian Banert was partially supported by ELLIIT.

\printbibliography

\end{document}